\documentclass[11pt,reqno, a4paper]{amsart}
\usepackage{graphicx} 
\usepackage{mathrsfs}
\usepackage{color}
\usepackage{epsfig,subfigure}
\usepackage[outdir=./]{epstopdf}
\usepackage{enumerate}
\usepackage{amssymb,amsmath,amsthm,graphicx,amsfonts}
\usepackage{hyperref}
\usepackage[margin=1.0in]{geometry}
\usepackage[ruled,linesnumbered]{algorithm2e}

\usepackage{titlesec} 
\titleformat{\section}{\vskip10pt\large\bfseries}{\thesection.}{0.5em}{\centering\vspace{5pt}}
\titleformat{\subsection}{\vskip10pt\normalsize\bfseries}{\thesubsection.}{0.5em}{}

\newtheorem{theorem}{Theorem}[section]

\newtheorem{lemma}[theorem]{Lemma}
\newtheorem{proposition}[theorem]{Proposition}
\newtheorem{remark}{Remark}[section]

\theoremstyle{definition}

\def\u{\tilde u(t_n+s)}
\def\v{\tilde v(t_n+s)}
\def\R{\mathbb{R}}

\allowdisplaybreaks

\numberwithin{equation}{section}

\begin{document}
\title[]{Numerical approximation of discontinuous solutions\\ of the semilinear wave equation}

\author[]{Jiachuan Cao\,\,}

\author[]{Buyang Li\,\,}
\address{\hspace*{-12pt}Jiachuan Cao, Buyang Li, Yanping Lin, Fangyan Yao: 
Department of Applied Mathematics, The Hong Kong Polytechnic University,
Hung Hom, Hong Kong. 
{\rm Email address: jiachuan.cao@polyu.edu.hk, buyang.li@polyu.edu.hk, yanping.lin@polyu.edu.hk, fangyan.yao@polyu.edu.hk}}
%

\author[]{Yanping Lin\,\,}

\author[]{Fangyan Yao\,\,}




\maketitle
\vspace{-20pt}

\begin{abstract}
A high-frequency recovered fully discrete low-regularity integrator is constructed to approximate rough and possibly discontinuous solutions of the semilinear wave equation. The proposed method, with high-frequency recovery techniques, can capture the discontinuities of the solutions correctly without spurious oscillations and approximate rough and discontinuous solutions with a higher convergence rate than pre-existing methods. Rigorous analysis is presented for the convergence rates of the proposed method in approximating solutions such that $(u,\partial_{t}u)\in C([0,T];H^{\gamma}\times H^{\gamma-1})$ for $\gamma\in(0,1]$. For discontinuous solutions of bounded variation in one dimension (which allow jump discontinuities), the proposed method is proved to have almost first-order convergence under the step size condition $\tau \sim N^{-1}$, where $\tau$ and $N$ denote the time step size and the number of Fourier terms in the space discretization, respectively. Numerical examples are presented in both one and two dimensions to illustrate the advantages of the proposed method in improving the accuracy in approximating rough and discontinuous solutions of the semilinear wave equation. The numerical results are consistent with the theoretical results and show the efficiency of the proposed method.  \\

\end{abstract}
%
\section{Introduction}\label{model} 
This article concerns the construction and analysis of numerical methods for approximating rough and possibly discontinuous solutions of the semilinear wave equation  
\begin{equation}\label{KG_eq}
\begin{cases}
\partial_{t t} u-\Delta u=g(u) & \text { in } (0, T] \times \Omega 
\\[2mm] \left.u\right|_{t=0}=u^0 \text { and }\left.\partial_t u\right|_{t=0}=v^0 & \text { in } \Omega 
\end{cases}
\end{equation}
in a domain $\Omega = [0,1]^d$ with the periodic boundary condition (i.e., $\Omega$ is regarded as a $d$-dimensional torus), where $g: \mathbb{R} \rightarrow \mathbb{R}$ is a given nonlinear function. For example, equation \eqref{KG_eq} is often referred to as the sine--Gordon equation when $g(u)=\sin(u)$, and often referred to as the nonlinear Klein--Gordon equation when $g(u)=-mu-\lambda u^3$; see \cite{CKW2014, Davydov2013,GI92}. Since the waves described by the semilinear wave equation propagate with finite speed, the problem in the whole space with compactly supported initial values can also be reduced to a bounded rectangular domain (with periodic boundary condition) which contains the support of the solution on the whole time interval $[0,T]$. 

As the relativistic version of the Schr\"{o}dinger equation, the semilinear wave equation has been wildly used in many physical areas such as quantum field theory, nonlinear optics, dislocated crystals, etc. During the last few decades, the numerical approaches for solving the semilinear wave equation have been extensively investigated, such as trigonometric/exponential integrators that are based on the variation-of-constants formula (for example, see
\cite{BDH2021,GH06,HL1999,G15,WW2019}), splitting methods (for example, see \cite{BFS2022,BY2007,BDH2021,GQ2014,HO2016,QG2016}), symplectic methods \cite{CHSS2020,CHL2008,HL00}, and finite difference methods (such as the Crank--Nicolson, Runge--Kutta and Newmark methods, see \cite{CLZ2021,HL2021,LLT2016,LS2020,LV2006,MK2011,QW2019,RX2008,WW2019-IMA}). These classical numerical methods have been shown convergent with optimal order to the sufficiently smooth solutions of the semilinear wave equation. However, due to the dispersion feature of \eqref{KG_eq}, roughness of the solution may be brought in by randomness or discontinuity of the initial data. This would cause significant challenges in constructing convergent numerical methods for approximating rough solutions, possibly discontinuous and with unbounded energy, of the semilinear wave equation. 

Many recent efforts were devoted to the construction and analysis of low-regularity integrators for nonlinear dispersive equations, such as the KdV equation \cite{Hofmanova-Schratz-2017,Li-Wu-Yao-2021,Wu-Zhao-IMA,Wu-Zhao-BIT} and the nonlinear Schr\"odinger equation \cite{BS,OS18,ORS20,RS21}. In these articles, several different approaches have been developed for constructing numerical methods that are convergent under low-regularity conditions, with higher-order convergence than the classical methods, including the resonance-based approach which uses variation-of-constants formulae and twisted variables \cite{BS,Hofmanova-Schratz-2017,OS18,OY2022,Li-Wu-Yao-2021,Wu-Yao-MCOM,Wu-Zhao-IMA,Wu-Zhao-BIT}, the semigroup based technique using the cancellation structures in the solution representations \cite{RS21,LMS}, and low-regularity integrators based on discrete Bourgain/Strichartz estimates \cite{ORS20}. Recently, based on new schemes to approximate the nonlinear frequency interaction and a new harmonic analysis technique by using the Littlewood-Paley dyadic decomposition, first-order low-regularity schemes for the cubic nonlinear Schr\"{o}dinger equation were introduced in \cite{LW2021} and \cite{BLW2022} to allow almost first-order convergence in the $L^{2}$ norm for $H^{1}$ initial data with periodic and Neumann boundary conditions respectively. Moreover, based on a temporal averaging technique and more careful high order resonance analysis, a new second-order scheme was proposed in \cite{CLL2023}, which can have second-order convergence in the $L^{2}$ norm with initial data strictly below $H^{2}$. These newly developed approaches have significantly improved the convergence rates of numerical solutions to these nonlinear dispersive equations under low regularity conditions. 

Low regularity integrators for the semilinear wave equation was addressed by Rousset \& Schratz in \cite{RS21} by using transformation $w=u-i(-\Delta)^{-\frac12} \partial_{t} u$, which converts the semilinear wave equation into a first-order formulation, i.e., 
\begin{align}
\label{formula:rootrewriting}
i \partial_{t} w=-(-\Delta)^{\frac12} w+(-\Delta)^{-\frac12} g\Big(\frac{w+\bar{w}}{2}\Big) . 
\end{align}
They constructed a low-regularity integrator for the first-order equation in \eqref{formula:rootrewriting} with 
second-order convergence in the energy norm $H^{1}\times L^{2}$ under the regularity condition $(u^{0},v^{0})\in H^{\frac{7}{4}}\times H^{\frac{3}{4}}$ in three dimensions. 

A different low-regularity integrator for \eqref{KG_eq} was constructed in \cite{LSZ} directly based on the discovery of a new cancellation structure, also with second-order convergence in $H^{1}\times L^{2}$ under the regularity condition $(u^{0},v^{0})\in H^{1+\frac{d}{4}}\times H^{\frac{d}{4}}$ for spatial dimension $d=1,2,3$. For the nonlinear term $g(u)=mu+\lambda u^{3}$ with given constants $m\geq 0$ and $\lambda\in \mathbb{R}$, a symmetric low-regularity integrator was constructed in \cite{WZ2022} for the semilinear Klein--Gordon equation on a one-dimensional torus, with second-order convergence in $H^{\gamma}\times H^{\gamma-1}$ under the condition $(u^{0},v^{0})\in H^{\gamma}\times H^{\gamma-1}$ for $\gamma>\frac{1}{2}$. 

These low-regularity integrators all require the solution $(u,\partial_tu)$ to be in $H^{\gamma}\times H^{\gamma-1}$ with $\gamma>\frac{d}{2}$, thus requiring the solution $u$ to be continuous. However, the semilinear wave equation can be well-posed even for discontinuous solutions below the energy space. For example, the sine--Gordon equation is well-posed in $H^{\gamma}\times H^{\gamma-1}$ for all $\gamma\ge 0$ and the semilinear Klein--Gordon equation is well-posed in $H^{\gamma}\times H^{\gamma-1}$ for $\gamma>\frac{1}{3}$ in the one dimensional case \cite{Bourgain1999}, and for $\gamma\geq 1$ in the high dimensional cases \cite{GV1985,GV1989}. The construction and analysis of convergent numerical methods for approximating such rough and discontinuous solutions of the semilinear wave equation are still interesting and challenging. 
%

The aim of this paper is to construct an efficient fully discrete low-regularity integrator for approximating possibly discontinuous solutions of the semilinear wave equation in one- and two-dimensional cases. To improve the convergence rates of numerical solutions in approximating rough solutions, we design a numerical scheme which could approximate the low-frequency and high-frequency parts of the solution separately, by approximating the low-frequency part with a time-stepping scheme and the high-frequency part with a recovery technique. The high-frequency recovery, which has equivalent computational cost as the approximation to the low-frequency part, could significantly improve the accuracy of the numerical solutions and therefore could capture the discontinuities in the rough solution by significantly reducing spurious oscillations. The advantages of the proposed method are demonstrated numerically in Section \ref{Section:Numerical} and proved rigorously in a particular setting, for approximating discontinuous solutions of bounded variation (such as piecewise smooth solutions) in one dimension. 

By utilizing the cancellation structure in the variation-of-constants-formula for the semilinear wave equation, and the new techniques developed in this paper, we prove the following error bounds in the $L^2\times H^{-1}$ norm for approximating a solution $(u,\partial_{t}u)\in C([0,T];H^{\gamma}\times H^{\gamma-1})$: 
\begin{align*}
& O(N^{-\gamma+}+\tau N^{1-2\gamma+}) &&\mbox{for}\,\,\,\gamma\in \Big(0,\frac{1}{2} \Big], \\
& O\big(N^{-\gamma+}+\tau N^{1-2\gamma+} +\min(\tau,\tau^{2}N^{2(1-\gamma)+})\big)
&&\mbox{for}\,\,\,\gamma\in \Big(\frac{1}{2},1 \Big] ,
\end{align*}
where $N$ denotes the number of Fourier terms used in each dimension of the space discretization. Therefore, the error is $O(\tau^\gamma)$ under the step size condition $\tau\sim N^{-1}$, and the convergence rate could be further improved by choosing a different step size condition which depends on the regularity of the solution. 

More importantly, we prove that for discontinuous solutions with bounded variation in one dimension (e.g., piecewise smooth solutions with jump discontinuities) the proposed numerical scheme has better convergence rate (i.e., almost first-order convergence) in $L^2\times H^{-1}$ under the step-size condition $\tau \sim N^{-1}$. 

Extensive numerical experiments, including both one- and two-dimensional examples, are given to illustrate the effectiveness (higher-order accuracy and reduction of spurious oscillation) of the proposed method in approximating rough and discontinuous solutions of the semilinear wave equation. 

The rest of this article is organized as follows. In Section \ref{section:main-results}, we introduce some basic notations and present the main theoretical results of this article. In Section \ref{section:technical} we present some preliminary and technical results have will be used in the construction and analysis of the method. 
The construction of the numerical scheme is presented in Section \ref{sec:scheme}. 
The proof of the main theorem and the improved results (for discontinuous BV solutions) are presented in Section \ref{section:proof} and Section \ref{section:proof2}, respectively. Finally, we provide numerical examples in Section 7 to support the theoretical results proved in this article and to illustrate the effectiveness of the proposed method in approximating rough/discontinuous solutions of the semilinear wave equation. 

\section{Notations and main results}\label{section:main-results}
In this section we introduce the basic notations and the numerical scheme, and then present the main theoretical results of this article on the convergence of numerical approximations to rough solutions of the semilinear wave equation. 

\subsection{Notations and numerical scheme}
 
We rewrite the second-order semilinear wave equation in \eqref{KG_eq} into a first-order system of equations, i.e.,  
\begin{align}\label{eq-mv}
\left\{
    \begin{aligned}
    &\partial_t U-LU=F(U) \ \text{in} \,\, (0,T]\times \Omega,\\ 
    &U(0)=U^0 \ \text{in} \ \Omega,
    \end{aligned}\right.
\end{align}
with  
\begin{align*}
    U=\left(\begin{array}{c} u \\ \partial_t u
\end{array}\right),
\quad 
U^0=\left(\begin{array}{c} u^0 \\ v^0
\end{array}\right),
\quad 
F(U)=\left(\begin{array}{c} 0 \\ f(u)
\end{array}\right),\quad 
L=\left(\begin{array}{ll}
0 & 1 \\ \widetilde{\Delta} & 0
\end{array}\right),
\end{align*}
where we have used the following notations: 
$$f(u)=g(u)+mu\quad \text{and}\quad\widetilde{\Delta}=\Delta-m \,\,\,\mbox{with some fixed constant $m>0$}.$$ 
The fixed constant $m>0$ is introduced to make sure that the linear operator $L$ is reversible on the $d$-dimensional torus $\Omega=[0,1]^d$. 

We denote by $e^{t L}$ the solution operator of the linear wave equation (i.e., the map from $U^0$ to $U(t)$ in the case $F(U)\equiv 0$). By defining the $a$-norm of a function $W=(w_1,w_2)^T \in H^a(\Omega)\times H^{a-1}(\Omega)$, i.e., 
\begin{align*}
    \|W\|_a=(\|w_1\|_{H^a}^2+\|w_2\|_{H^{a-1}}^2)^\frac12,\quad a\in\R . 
\end{align*}
the following properties hold: 
\begin{align}\label{propL}
&\| e^{t L}W\|_{0} \lesssim \| W\|_0,\quad \| e^{t L} W\|_{1} \lesssim \| W\|_1,\\
& \| F(U)\|_1\lesssim \| f(u)\|_{L^2}\lesssim f(0)+\|f^{\prime}\|_{L^{\infty}}\|U\|_{0}.
\end{align}

It is known that any function in the Sobolev space $L^2(\Omega)$ can be expanded into a Fourier series. Accordingly, we introduce the finite-dimensional subspace 
$$
S_N=\bigg\{\sum_{n_1, \cdots, n_d=-N}^N c_{n_1, \cdots, n_d} \exp \left(2n_1 \pi x_1 i\right) \cdots \exp \left(2n_d \pi x_d i\right): c_{n_1, \cdots, n_d} \in \mathbb{C}\bigg\} ,
$$
and approximate functions in $H^s(\Omega)$ by using the finite-dimensional subspace $S_N$. We denote by $\Pi_N$ the $L^2$ projection operator onto $S_N$ defined by 
\begin{align*}
(w-\Pi_Nw, v) = 0, \quad \forall v\in S_N, \quad w\in H^s(\Omega) , 
\end{align*}
and denote $ \Pi_{>N}:= I - \Pi_{N}$ and $\Pi_{(N_1,N_2]}:=\Pi_{N_2}-\Pi_{N_1}$ for $N_2>N_1$. 
We denote by $I_N: H^s\to S_N$ the trigonometric interpolation  such that  
for any function $w\in H^s$, $s>\frac{d}2$,
$\left(I_N w\right)(x)=w(x)$ for $x \in D^d$, with
$$
D=\left\{\frac{n}{2 N}: n=0, \cdots, 2N-1\right\}. 
$$

We define the sequence of grid points $t_n=n\tau$, $n=0,1,\dots, M$, in the time interval $[0,T]$ with step size 
$\tau=T/M$, and denote by $U^n=(u^n,v^n)^T$ the numerical solution at time $t=t_n$. 
Then the high-frequency recovering low-regularity integrator for the equation \eqref{eq-mv} constructed in this article reads (the detailed construction is presented in Section \ref{sec:scheme}): 
\begin{subequations}\label{reformulated_scheme}
\begin{align}
U^{n+1}_{N}&=\Pi_N U^{n+1}_{N}+\Pi_{(N,N^{\alpha}]} U^{n+1}_{N},\label{HR-LRI_1}\\[2mm]
\Pi_N U^{n+1}_{N}&=e^{ \tau L} \Pi_N U^{n}_{N} + \tau e^{ \tau L} I_N F( \Pi_N U^{n}_{N}) \notag\\[2mm]
&\quad+ (2L)^{-1} \left[
\tau e^{ \tau L} - (2L)^{-1} (e^{ \tau L} - e^{- \tau L} )
H(\Pi_N U^{n}_{N})
\right],\label{HR-LRI_2}\\[2mm]
\Pi_{(N,N^{\alpha}]} U^{n+1}_{N}&=e^{ (n+1)\tau L}\Pi_{(N,N^{\alpha}]} U^{0}_{N},\label{HR_LRI_3}
\end{align}
\end{subequations}
with 
\begin{align}\label{defH}
H(U)=
 \begin{pmatrix}
	-I_N f(u)\\
	\Pi_N ( I_N f^\prime (u)\cdot v)
\end{pmatrix},
\end{align}
and initial value $U_{N}^{0}=\Pi_{N^{\alpha}}U(0)$ for $\alpha\geq 1$. This scheme introduces an additional high-frequency part to the low-regularity integrator in \cite{LSZ}, and uses a different definition of $H(U)$ to obtain the desired convergence rates for approximating discontinuous solutions. It is the combination of the filtered low-regularity integrator \eqref{HR-LRI_2} and high-frequency recovery process \eqref{HR_LRI_3}, ensures the accuracy of the proposed method for approximating rough solutions under lower-regularity conditions than \cite{LSZ}. 

\begin{remark}\label{Remark2.1}\upshape 
In \eqref{reformulated_scheme} we see that the high- and low-frequency parts of $U^{n+1}_{N}$ are computed separately, independent of each other, where the low-frequency part is computed by a time-stepping scheme and the high-frequency part is recovered directly from its initial value via \eqref{HR_LRI_3} (without time steppings). 
Since the nonlinear terms in \eqref{HR-LRI_2} can be computed by using the Fast Fourier Transform (FFT), the computational cost at every time level is $O(N^{d}\log(N))$. Therefore, the total cost for computing the low-frequency part at time $T$ is $O(N^{d}\log(N) T/\tau )$. In contrast, the high-frequency part of the numerical solution needs not be computed every time level. Instead, we only need to compute \eqref{HR_LRI_3} once to recover the high-frequency part of $U^{n+1}_{N}$ for any particular time level of interest. Therefore, the cost of computing the high-frequency part at time $T$ is $O(N^{\alpha d})$, which is comparable to the cost of computing the low-frequency part if we choose $N^{(\alpha-1)d}\sim T/\tau$. Under the step size condition $\tau\sim N^{-1}$, this suggests to choose $\alpha=1+\frac1d$ in the  computation. The advantages of this choice is analyzed rigorously in Theorem \ref{u11} in one dimension and illustrated numerically in Section \ref{Section:Numerical} in both one and two dimensions. 

\end{remark}

\begin{remark}\upshape 
The trigonometric interpolation operator $I_N$ is used in \eqref{HR-LRI_2} and \eqref{defH} to approximate nonlinear functions by Fourier series using FFT. This makes the algorithm more efficient than using the projection operator $\Pi_N$, but also increases the difficulty of convergence analysis under low-regularity conditions. 
\end{remark}

In the rest of this paper, we show that the combination of the filtered low-regularity integrator \eqref{HR-LRI_2} and high-frequency recovery process \eqref{HR_LRI_3}, ensures the stability and accuracy of the proposed method (which eliminates spurious oscillations in the numerical solution).




\subsection{Main theoretical results}

For the simplicity of notation, we denote by $A\lesssim B$ or $B\gtrsim A$ the statement $A\leq CB$ for some constant $C>0$. The value of $C$ may depend on $T$ and $\|U\|_{\gamma}$, and may be different at different occurrences, but is always independent of step size $\tau$, degrees of freedom $N$ (in each dimension), and time level $n$. The notation $A\sim B$ means that $A\lesssim B\lesssim A$. If a statement contains $s+$ or $s-$ for some number $s$, it means that the statement holds with $s+\epsilon$ or $s-\epsilon$ for arbitrary $\epsilon>0$; see Theorem \ref{thm:convergence}. 

The convergence of the proposed algorithm in \eqref{reformulated_scheme} in the general setting is presented in the following theorem.


\begin{theorem}\label{thm:convergence} 
Let $d=1,2$ and $\gamma\in (0,1]$, and assume that the nonlinear function $f:\R\rightarrow\R$ satisfies the following condition: 
 \begin{align}\label{cond:f}
     |f'(s)| + |f''(s)| + |f'''(s)| \lesssim 1 . 
\end{align}
Then, under the regularity condition $U \in C([0,T]; H^\gamma(\Omega) \times H^{\gamma-1}(\Omega))$ and the step size condition $N\lesssim \tau^{-\frac{1}{1- \gamma}+\epsilon_{0}}$ {\rm(}for an abitrary $\epsilon_{0} \in(0,1]${\rm)}, the numerical solutions given by \eqref{reformulated_scheme} for each $\alpha\geq 1$ converge to the solution of \eqref{eq-mv} with the following error estimates. 
\begin{itemize}
 \item[(i)]For $\gamma\in(0,\frac12]$, there exist constants $ \tau_0\in (0,1)$ and $C_0>0$ such that, for step size $ \tau\in (0, \tau_0]$, 
\begin{equation}\label{eq:thm_2}
\max _{0 \leq n \leq T / \tau}\left\|U(t_n)-U^n_N\right\|_0 \le C_0(N^{-\gamma+}+\tau N^{1-2\gamma+}) .
\end{equation}

\item[(ii)]For $\gamma\in(\frac12,1]$, there exist constants $ \tau_0\in (0,1)$ and $C_0>0$ such that, for step size $ \tau\in (0, \tau_0]$, 

\begin{equation}\label{eq:thm_1}
\max _{0 \leq n \leq T / \tau}\left\|U(t_n)-U^n_N\right\|_0 \le C_0( N^{-\gamma+}+\tau N^{1-2\gamma+}+\min(\tau,\tau^{2}N^{2(1-\gamma)+})) .
\end{equation}
 \end{itemize}
The constants $C_0$ and $\tau_0$ may depend on $\epsilon_0$ in the condition $N\lesssim \tau^{-\frac{1}{1- \gamma}+\epsilon_{0}}$ {\rm(}when $\epsilon_0$ is smaller, $\tau_0$ is smaller and $C_0$ is bigger{\rm)}.  
\end{theorem}

\begin{remark}\upshape
	Under the condition \eqref{cond:f}, by constructing a contraction map, standard techniques can be used to prove that problem \eqref{eq-mv} admits a unique solution $U=(u,\partial_{t}u)\in L^{\infty}(0,T;H^{\gamma}\times H^{\gamma-1})$. This solution is automatically in $C([0,T];H^\gamma\times H^{\gamma-1})$. 
\end{remark}


\begin{remark}\upshape 
The theoretical error estimates in Theorem \ref{thm:convergence} implies that, by choosing $\tau\sim N^{-1}$ (independent of the regularity of the initial data) and $\alpha= 1$ (without high-frequency recovery), the errors of the numerical solutions is bounded by $O(\tau^{\gamma-})$ for approximating solutions in $H^\gamma(\Omega)\times H^{\gamma-1}(\Omega) $ with $\gamma \in (0,1]$. However, in practical computation, the errors of the numerical solutions can often be significantly reduced by using high-frequency recovery with $\alpha=1+\frac{1}{d}$ under the step size condition $\tau\sim N^{-1}$ with equivalent computational cost; see Remark \ref{Remark2.1}. Such advantages of the high-frequency recovery technique proposed in \eqref{reformulated_scheme} is demonstrated numerically in Section \ref{Section:Numerical} and proved rigorously in the following theorem in a particular setting, for approximating discontinuous solutions of bounded variation (such as piecewise smooth solutions) in one dimension. 

\end{remark} 


Let $BV(\Omega)$ denote the set of functions with bounded variations on $\Omega$ with norm 
$ \|u\|_{BV} := \|u\|_{L^1} + \|\nabla u\|_{M} , $ where $M$ denotes the norm of $M(\Omega)$, i.e., the space of Borel measures on $\Omega$ (the norm of $M(\Omega)$ is equivalent to the $L^1$ norm for integrable functions). 

\begin{theorem}\label{u11} 
Let $d=1$ (i.e., consider the one-dimensional problem) and assume that the solution has the following regularity: 
$$
(u,\partial_tu) \in C([0,T]; H^{\frac12-}(\Omega) \times H^{-\frac12-}(\Omega)) 
\quad\mbox{and}\quad 
u \in L^\infty(0,T; BV(\Omega)\cap L^{\infty}(\Omega))  .
$$ 
Then, under the step size condition $\tau\sim N^{-1}$, the numerical solutions given by \eqref{reformulated_scheme} with $\alpha=2$ converge to the solution of the continuous problem in \eqref{eq-mv} with the following rate:
\begin{align}\label{improved_result} 
\max _{0 \leq n \leq T / \tau}\left\|U(t_n)-U^n_N\right\|_0 
\lesssim \tau^{1-} . 
\end{align}
\end{theorem}

\begin{remark}\label{resin}\upshape
For an initial value in $(u_0,v_0) \in BV(\Omega) \times M(\Omega)$, the additional regularity condition $u\in L^\infty(0,T; BV(\Omega) \cap L^{\infty}(\Omega))$ in Theorem \ref{u11} naturally holds for the one-dimensional sine-Gordon equation  
    \begin{align}\label{1d-sG}
\left\{
    \begin{aligned}
    &\partial_{tt} u-\partial_{xx}u=\sin(u) &&\mbox{in}\,\,\,(0,T]\times \Omega \\ 
    &u|_{t=0}=u_0,\quad \partial_tu|_{t=0}=v_0 &&\mbox{on}\,\,\,\Omega .
    \end{aligned}\right.
\end{align}
For example, we consider \eqref{1d-sG} with initial value $u_0\in BV(\Omega)$ on the interval $\Omega=[0,1]$ with the periodic boundary condition. 
Let $\tilde{u}$, $\tilde{u}_0$ and $\tilde{v}_0$ be the periodic extensions of $u$, $u_0$ and $v_0$ to $\mathbb{R}$, respectively. Then d'Alambert's formula and Duhamel's formula imply that 
\begin{align}\label{ddformu}
\tilde u(t,x)=\frac12\int_0^t \int_{x-t+s}^{x+t-s} \sin(\tilde u)(s,y) dyds
          +\frac12 (\tilde u_0(x+t)+\tilde u_0(x-t))
          +\frac12 \int_{x-t}^{x+t} \tilde v_0(y)dy . 
\end{align}
By taking the $BV$ and $L^\infty$ norm on both sides of \eqref{ddformu} and then summing up the two results, we can see that $u(t) \in BV(\Omega)$ for all $t\in[0,T]$ and 
\begin{align*}
    \|u(t)\|_{BV}+\|u(t)\|_{L^{\infty}} \lesssim T^{2} +\|u_0\|_{BV}+\|u_0\|_{L^{\infty}}+ \|v_0\|_{M}.
\end{align*}
This shows that the regularity condition in Theorem \ref{u11} naturally characterizes the regularity of the solutions with discontinuous initial data in $BV(\Omega)$. 

\end{remark}



\begin{remark}\upshape 
Without the high-frequency recovery, the convergence rate of the proposed method would reduce to half order from first order. This can be seen from the proof of Theorem \ref{u11} and can also be observed in the numerical tests. 
\end{remark}

\begin{remark}\upshape 
From the numerical examples in Section 7.2 we can see that the high-frequency recovery in \eqref{reformulated_scheme} also significantly improves the convergence rates of the numerical solution in two dimensions. We present numerical tests for approximating rough/discontinuous solutions by taking $\alpha=\frac{3}{2}$ in \eqref{reformulated_scheme} under the step size condition $\tau \sim N^{-1}$. The numerical results show that the proposed method has a convergence rate of order $\frac{3}{4}$, which is distinctively better than existing numerical methods.
\end{remark}

\section{Preliminary results}\label{section:technical}

In this section we present some preliminary results to be used in the proofs of Theorem \ref{thm:convergence} and Corollary \ref{u11}. These include Bernstein's inequalities in the $L^p$ norm (Lemma \ref{bound}), approximation properties of the trigonometric interpolation  (Lemma \ref{lemInterp}), $L^{p}$ error of trigonometric interpolation (Lemma \ref{lemInterp_2}), and negative-norm estimates for the product of two functions (Lemma \ref{Lemma:negative-norm} and Lemma \ref{lem:H_estimate}).  

\begin{lemma}[{Bernstein's inequality; cf. \cite[Theorem 2.2 and pp. 22]{Guo1998}}]\label{bound}
Let $f$ be a function such that $J^\gamma f:= (1- \Delta)^\frac \gamma2 f \in L^p( \Omega)$ for some $\gamma\ge0$ and $1<p<\infty$.
Then the following results hold:
\begin{align*}
 \| \Pi_{\le N}J^\gamma f\|_{L^p}
\lesssim
N^{\gamma}
\| f\|_{L^p},\qquad
\| \Pi_{> N} f\|_{L^p}
\lesssim
N^{-\gamma}
\|J^{\gamma}  f\|_{L^p}.
\end{align*}
\end{lemma}

\begin{lemma}[{Standard error of trigonometric interpolation; cf. \cite[Theorem 11.8]{Kress1989}}]\label{lemInterp}
Let $f$ be a function such that $f\in H^\gamma( \Omega)$.
For $0\leq s\leq \gamma$ and $ \gamma> \frac{d}{2}$, we have 
\begin{align*}
\| f - I_N f\|_{H^{s}} \lesssim N^{- (\gamma-s)} \| f\|_{H^\gamma}.
\end{align*}

\end{lemma}

\begin{lemma}[{Error of trigonometric interpolation in the $L^{p}$ norm}]\label{lemInterp_2} Let $d=1$ and $f\in W^{1,p}(\Omega)$ for $1<p<\infty$ then
\begin{align*}
\| f - I_N f\|_{L^{p}} \lesssim N^{- 1} \| f\|_{W^{1,p}}.
\end{align*}
\end{lemma}

\begin{proof}
Using the trigonometric interpolation error estimates in \cite[Theorem 3]{Hristov1989} and \cite[Lemma 2.3]{Prestin1994}, one get
\begin{align*}
\| f - I_N f\|_{L^{p}} \lesssim N^{-1} E_n(f^{\prime})_p,
\end{align*}
where $E_n(f^{\prime})_p$ is the best approximation of $f^{\prime}$ with trigonometric polynomials from $S_N$ in the $L^{p}$ norm:
$$E_n(f^{\prime})_p:=\inf \{\|f^{\prime}-T\|_{L^{p}}: T\in S_N\}\leq \|f^{\prime}-\Pi_{\leq N} f^{\prime}\|_{L^{p}}= \|\Pi_{> N} f^{\prime}\|_{L^{p}}\lesssim \|f\|_{W^{1,p}}.$$
\end{proof}

\begin{lemma}[Negative-norm estimates for the product of two functions]\label{Lemma:negative-norm}
For $d = 1,2$, the following estimates hold: 
\begin{align}
\label{dualargu2}
\| fg\|_{H^{-1}}&\lesssim \| g\|_{H^{-1}} \left( \|f\|_{L^\infty} +\| f\|_{H^{1+}}\right),\\
\label{embedding3}
\| fg\|_{H^{-1}}&\lesssim \| f\|_{L^{2+}}^ \gamma \| f\|_{H^{1+}}^ {1- \gamma} \| g\|_{H^{ \gamma -1}},\\
\label{embedding2}
\| fg\|_{H^{-1}}&\lesssim \| f\|_{L^2} \| g\|_{L^{2+}}.
\end{align}
In addition, for any $\gamma\in (0,1]$ and function $h\in L^\infty$, the following estimate holds: 
\begin{align}
\label{embedding}
\| fgh\|_{H^{\gamma -1}}&\lesssim \| f\|_{L^2} \| g\|_{H^ {\gamma+}} \| h\|_{L^\infty}.
\end{align} 
\end{lemma}
\begin{proof}
By the dual argument and Sobolev embedding in one and two dimensions (i.e., $H^{1+}\hookrightarrow L^\infty$), we have
\begin{align*}
\| fg\|_{H^{-1}} &= \sup_{ \| w\|_{H^{1}}=1} 
\langle g, fw \rangle \lesssim  \sup_{ \| w\|_{H^{1}}=1} \| g\|_{H^{-1}} \| fw\|_{H^{1}}\\
&\lesssim \sup_{ \| w\|_{H^{1}}=1} \| g\|_{H^{-1}} \left( \|f\|_{L^\infty} \| w\|_{H^{1}} + \| w\|_{L^{\infty-}} \| f\|_{H^{1+}}\right)\\
&\lesssim \sup_{ \| w\|_{H^{1}}=1} \| g\|_{H^{-1}} \left( \|f\|_{L^\infty} +\| f\|_{H^{1+}}\right)\| w\|_{H^{1}}.
\end{align*}
This completes the proof of \eqref{dualargu2}. Similarly, for \eqref{embedding3}, we have 
\begin{align}\label{proof0}
\| fg\|_{H^{-1}}&= \sup_{ \| w\|_{H^1} = 1} \langle g, fw \rangle
\lesssim \sup_{\| w\|_{H^1} = 1} \| g\|_{H^{ \gamma - 1}} \| fw\|_{H^{1- \gamma}}.
\end{align}
By the interpolation inequality and embeddings $H^1\hookrightarrow L^{\infty -}$ and $H^{1+}\hookrightarrow L^\infty$, we obtain
\begin{align}\label{proof2}\notag
\| fw\|_{H^{1- \gamma}}\lesssim 
\| fw\|_{L^2}^ \gamma 
	\| fw\|_{H^1}^{1- \gamma}
&\lesssim 
	\| f\|_{L^{2+}}^ \gamma \| w\|_{L^{\infty -}}^ \gamma \left[
	\| J f\|_{L^{2+}}^{1- \gamma} \| w\|_{L^{\infty-}}^{1- \gamma} + \| f\|_{L^\infty}^{1- \gamma} \| J w\|_{L^2}^{1- \gamma}
	\right]\\
&\lesssim \| f\|_{L^{2+}}^ \gamma \| f\|_{H^{1+}}^{1- \gamma} \| w\|_{H^1}.
\end{align}
Substituting \eqref{proof2} into \eqref{proof0} leads to \eqref{embedding3}.

Estimate \eqref{embedding2} can be verified by using the Sobolev embedding result 
$L^{1+}\hookrightarrow H^{-1}$ in one- and two-dimensional spaces.
 
For \eqref{embedding},    the Sobolev embedding  $L^q\hookrightarrow H^{ \gamma - 1}$ with  $ \gamma\in (0,1]$,
$ 1/q = 1/2 - (\gamma-1)/d $
leads to the following inequalities 
\begin{align*}
\| fgh\|_{H^{ \gamma -1}}\lesssim \| fg\|_{L^q} \| h\|_{L^\infty}
\le \left\{\begin{array}{ll}
f\|_{L^2} \| g\|_{L^\frac{d}{1- \gamma}}\| h\|_{L^\infty},\quad &\text{for} \quad \gamma\in (0,1),\\[2mm]
\| f\|_{L^2} \| g\|_{L^\infty}\| h\|_{L^\infty},\quad &\text{for}\quad \gamma=1.
\end{array}
\right.
\end{align*}
Then, the Sobolev embedding 
$H^{\gamma}\hookrightarrow L^\frac{d}{1-\gamma}$ and $H^{1+}\in L^{\infty}$ yields
\begin{align*}
    \| fgh\|_{H^{ \gamma -1}}&\lesssim \| f\|_{L^2} \| g\|_{H^ {\gamma+}} \| h\|_{L^\infty}.
\end{align*} 
This proves the desired results.
\end{proof}

\begin{lemma}\label{lem:H_estimate}
Let $(u,v)^{T}\in H^{\gamma}(\Omega)\times H^{\gamma-1}(\Omega)$ and assume that $f$ satisfies the conditions in \eqref{cond:f}. Then the following estimates hold: 
\begin{align}\label{gamma_leq_frac12}
\|f^{\prime}(\Pi_N u)\Pi_N v\|_{H^{-1}} &\lesssim N^{1-2\gamma+} && \text{for}\quad \gamma\in \mbox{$(0,\frac12]$}, \\ 
\label{gamma_greater_frac12}
\|f^{\prime}(\Pi_N u)\Pi_N v\|_{H^{-1}} &\lesssim 1 && \text{for}\quad \gamma\in \mbox{$(\frac12,1]$}.
\end{align}
\end{lemma}

\begin{proof}
Without loss of generality, we assume that $\log_2 N $ is an integer (otherwise we replace $\log_2 N $ by the smallest integer larger than it). Then, by using the triangle inequality, we have 
\begin{align}\label{star}
\| f^\prime (\Pi_N  u) \cdot \Pi_N  v \|_{H^{-1}}
&\le \sum_{k=1}^{\log_2 N-1} 
\| (f^\prime( \Pi_{2^{k+1}}  u) 
        - f^\prime( \Pi_{2^{k}}  u))
        \Pi_N  v\|_{H^{-1}}.
\end{align}
The right-hand side of \eqref{star} can be estimated by using \eqref{embedding3}, i.e., 
\begin{align*}
	&\quad \| (f^\prime( \Pi_{2^{k+1}}  u) 
	- f^\prime( \Pi_{2^{k}}  u))
	\Pi_N v\|_{H^{-1}}\\
	&\lesssim \| f'( \Pi_{2^{k+1}}  u) - f'( \Pi_{2^{k}}  u)\|_{L^{2+}}^ \gamma 
	\| f'( \Pi_{2^{k+1}}  u) - f'( \Pi_{2^{k}}  u)\|_{H^{1+}}^ {1-\gamma}
	\|\Pi_N  v \|_{H^{ \gamma -1}}\\
	&\lesssim\|f''\|_{L^\infty}^{\gamma}\|\Pi_{2^{k+1}}  u-\Pi_{2^{k}}  u\|_{L^{2+}}^{\gamma}\cdot\left(\| f'( \Pi_{2^{k+1}}  u) \|_{H^{1+}}+\| f'( \Pi_{2^{k}}  u)\|_{H^{1+}}
	\right)^ {1-\gamma}\|\Pi_N  v \|_{H^{ \gamma -1}}\\
	&\lesssim [(2^k)^{- \gamma +}]^ \gamma [(2^k)^{1- \gamma +}]^ {1-\gamma}
	\lesssim (2^k)^{1-2 \gamma +},
\end{align*}
here in the third inequality, we have used the boundedness of $\|f''\|_{L^\infty}$ and $\|\Pi_N  v \|_{H^{ \gamma -1}}$, and the following estimates: 
\begin{align*}
	\|\Pi_{2^{k+1}}  u-\Pi_{2^{k}}  u\|_{L^{2+}}&\lesssim \|\Pi_{2^{k+1}}  u-\Pi_{2^{k}}  u\|^{1-}_{L^{2}} \|\Pi_{2^{k+1}}  u-\Pi_{2^{k}}  u\|^{0+}_{H^1}\\
	&\lesssim \left(\big(2^{k}\big)^{-\gamma}\|u\|_{H^{\gamma}}\right)^{1-}\left(\left(2^{k}\right)^{1-\gamma}\|u\|_{H^{\gamma}} \right)^{0+}\lesssim (2^{k})^{-\gamma+},
\end{align*}
and
\begin{align*}
	\| f'( \Pi_{2^{k}}  u)\|_{H^{1+}}
	&\lesssim \| f'( \Pi_{2^{k}}  u) \|_{H^{1}}^{1-}\| f'( \Pi_{2^{k}}  u) \|_{H^{2}}^{0+}\\
	&\lesssim \left(\|f'\|_{L^{\infty}}+\|f''\|_{L^{\infty}}\|\Pi_{2^{k}}  u\|_{H^1}\right)^{1-}\\
	&\quad\cdot\left(\|f'\|_{L^{\infty}}+\|f''\|_{L^{\infty}}\|\Pi_{2^{k}}  u\|_{H^2}+\|f'''\|_{L^{\infty}}\left\|\left(\Pi_{2^k}\nabla u\right)^2\right\|_{L^2}\right)^{0+}\\
	&\lesssim \left(1+\left(2^{k}\right)^{1-\gamma}\|u\|_{H^{\gamma}}\right)^{1-}\left[1+\left(2^{k}\right)^{2-\gamma}\|u\|_{H^{\gamma}}+\left(\left(2^{k}\right)^{1+\frac{d}{4}-\gamma}\|u\|_{H^{\gamma}}\right)^2\right]^{0+}\\
	&\lesssim \left(2^{k}\right)^{1-\gamma+}.
\end{align*}
For $ \gamma\in (0,\frac12]$, we sum up the above estimate for $k=1,\dots,\log_2N -1$. This leads to the following result: 
\begin{align*}
 \| \Pi_N( f^\prime (\Pi_N \tilde u) \cdot \Pi_N \tilde v )\|_{H^{-1}}
&\lesssim \sum_{k=1}^{\log_2 N -1}
 (2^k)^{1-2 \gamma+}
 \lesssim N^{1-2 \gamma + } \cdot \log_2 N.
\end{align*} 
This proves the first result of Lemma \ref{lem:H_estimate}, where the term $\log_2N$ can be absorbed into the $N^{1-2 \gamma +} $ with the presence of the ``$+$'' in the exponent. 

For $\gamma \in (\frac12,1]$, we have $1-2\gamma+<0$ and therefore $ \| \Pi_N( f^\prime (\Pi_N \tilde u) \cdot \Pi_N \tilde v )\|_{H^{-1}}
\lesssim 1 $. 
This proves the second result of Lemma \ref{lem:H_estimate}. 
\end{proof}

Next, we present some useful properties/structures of the vector-valued function $F(U)$ that play important roles in the convergence of the proposed low-regularity integrator for approximating rough solutions below the energy space. Since $F(U) =(0, f(u))^T$, it is straightforward to verify that, for $W =(w_1, w_2)^T$ and $ W^* = (w_1^*, w_2^*)^T$,
\begin{align}\label{dF}
	F^\prime(U) = \begin{pmatrix} 0 & 0\\ f^\prime(u) & 0\end{pmatrix}, \quad 
	F^\prime(U)W = \begin{pmatrix}0\\ f^\prime(u)w_1\end{pmatrix},\quad F^{\prime\prime}(U)W\cdot W^* = \begin{pmatrix}0\\ f^{\prime\prime}(u) w_1w_1^*\end{pmatrix}.
\end{align}

The next lemma is a fundamental ingredient in the construction of the second-order low-regularity integrator proposed in \cite{LSZ}.
\begin{lemma}\label{mainlem}
Let $U = (u,v)^T$ and $\tilde U(s) := e^{sL} U = (\tilde u(s), \tilde v(s))^T$.
Then the following identities hold
\begin{align}\label{dG}
\dfrac{d}{d s} e^{- s L} F ( \tilde U(s) )
=
e^{-sL}\begin{pmatrix}
-  f(\tilde u(s))\\
f^\prime ( \tilde u(s)) \cdot  \tilde v(s)
\end{pmatrix},
\end{align}
and 
\begin{align}\label{ddG}
\dfrac{d}{ds} \left[e^{2s L} 
\dfrac{d}{ds} \big(e^{- s L}F(\tilde U( s))\big)\right] 
= e^{sL} \begin{pmatrix} 0\\ f^{\prime\prime}
	 (\tilde u(s))(|\tilde v(s)|^2 - |\nabla \tilde u(s)|^2)+m\left(\tilde{u}(s)-f(\tilde{u}(s))\right)
\end{pmatrix}.
\end{align} 
Moreover,
\begin{align}\label{I1form}\notag
\int_0^ \tau e^{ (\tau-s)L}F(\tilde U(s)) d s 
&= \tau e^{ \tau L} F(U) + (2L)^{-1} [ \tau e^{ \tau L} - (2L)^{-1} (e^{ \tau L} - e^{- \tau L})]
\begin{pmatrix}
-f(u)\\ f^\prime (u) v
\end{pmatrix}\\
&\quad+ \int_0^ \tau e^{ (\tau - 2 s) L} ( \tau -s) \int_0^s \dfrac{d}{d \sigma} 
\Big[e^{2 \sigma L} 
\dfrac{d}{d \sigma} \big(e^{- \sigma L}F(\tilde U( \sigma))\big) \Big] d \sigma ds.
\end{align}
\end{lemma}
Equation \eqref{ddG} is the cancellation structure in the semilinear wave equation discovered by \cite{LSZ}. This cancellation structure ensures that all spatial second-order derivatives on the right-hand side of \eqref{ddG} are canceled out, allowing us to construct suitable numerical approximations for the nonlinear term using Equation \eqref{I1form}, while avoiding higher-order derivative terms in the error analysis. The proof of the above lemma involves the application of the chain rule, the Fubini theorem, and integration by parts. For detailed proofs and further information, refer to \cite[(2.23) and (2.26)]{LSZ}. 
Furthermore, the following result, which can be found in \cite[(2.44)]{LSZ} will be useful later,
\begin{align}\label{inequdF}
\Big\| \dfrac{d}{d \sigma} e^{- \sigma L} \left[F^\prime (\tilde U(\sigma)) e^{\sigma L} W\right]\Big\|_1
\lesssim \| W\|_{\frac32+ \epsilon} \| \tilde U(\sigma)\|_1 
+ \| W\|_1.
\end{align}

\section{Construction of the numerical scheme}\label{sec:scheme}
In this section, we demonstrate the construction of the high frequency recovered low-regularity integrator in \eqref{reformulated_scheme}. 
We start with the variation-of-constants formula, i.e., 
\begin{align}\label{duhamel}
U(t_{n+1})
& =e^{\tau L} U\left(t_n\right)+\int_0^\tau e^{(\tau-s) L} F\left(U\left(t_n+s\right)\right) d s 
\end{align}
and employ the following high- and low-frequency decomposition: 
\begin{align}\label{1}
U\left(t_{n+1}\right)& =e^{\tau L} U\left(t_n\right)+\int_0^\tau e^{(\tau-s) L} \Pi_N F\left(\Pi_N U\left(t_n+s\right)\right) d s+ R_1(t_n),
\end{align}
where the remainder $R_1(t_n)$ is given by
\begin{align}\label{defR1}
R_1(t_n)=
\int_0^\tau e^{(\tau-s) L}
\left[F\left(U\left(t_n+s\right)\right)- \Pi_N F\left(\Pi_N U\left(t_n+s\right)\right)\right] d s.
\end{align}
Then we approximate the low-frequency term. 

Since $ \Pi_N$ commutes with $L$, 
by using the Taylor expansion of $F(U)$ at $U = e^{sL} \Pi_N U(t_n)$, we have
\begin{align*}
F( \Pi_N U(t_n + s)) &= F(e^{sL} \Pi_N U(t_n)) + F^\prime(e^{sL} \Pi_N U(t_n))
	\int_0^s e^{(s- \sigma) L} \Pi_N F(\Pi_N U(t_n+ \sigma)) d \sigma \\
 &\quad+ \tilde R_2(s)
 + \tilde R_3(s),
\end{align*}
 where
 \begin{align}\label{def-tilde-R2}
 \tilde R_2(s)&= F^\prime(e^{sL} \Pi_N U(t_n))
	\int_0^s e^{(s- \sigma) L} \Pi_N
 \Big[ F( U(t_n+ \sigma)-F(\Pi_N U(t_n+ \sigma)) \Big]d \sigma,\\
 \label{def-tilde-R3}
\tilde R_3(s) &= R_F(s) \int_0^ s e^{(s- \sigma)L} \Pi_{N} F(U(t_n + \sigma)) d \sigma
	\cdot \int_0^s e^{(s- \sigma)L} \Pi_N F(U(t_n + \sigma)) d \sigma,\\
R_F(s) &= \int_0^1 \int_0^1 \theta F^{\prime \prime} \left[ (1- \sigma) e^{sL} \Pi_N U(t_n) + \sigma (1- \theta)
	e^{sL} \Pi_N U(t_n) + \theta \sigma \Pi_N U(t_n+ s)\right] d \sigma d \theta. \notag
\end{align}
Inserting the above results to  \eqref{1}, we obtain
\begin{align*}
U(t_{n+1}) = e^{ \tau L} U(t_n) + I_1(t_n) + I_2(t_n) + R_1(t_n)+R_2(t_n)+R_3(t_n),
\end{align*}
where
\begin{align}
I_1(t_n) &= \int_0^ \tau e^{ (\tau-s) L} \Pi_N F(\Pi_N \tilde U(t_n+s)) d s,\quad
\tilde U(t_n+s) = e^{sL} U(t_n),\label{I1_tildeU}\\\notag
I_2(t_n) &= \int_0 ^ \tau e^{(\tau-s)L} \Pi_N \left[ F^\prime(\Pi_N \tilde U(t_n+s))
				\int_0^s e^{(s- \sigma)L} \Pi_N F( \Pi_N U(t_n+ \sigma)) d \sigma\right] ds,\\
R_2(t_n) &= \int_0^ \tau e^{( \tau - s )L} \Pi_N \tilde R_2(s) ds,
\quad R_3(t_n) = \int_0^ \tau e^{( \tau - s )L} \Pi_N \tilde R_3(s) ds.\label{R2R3}
\end{align}
The approximation of $I_1(t_n)$ combines the filtering technique and Lemma \ref{mainlem}.
Using \eqref{I1form} with $U = \Pi_N U(t_n)$ and interchanging $ \Pi_N$ and $L$ in order, we obtain
\begin{align*}
I_1(t_n) 
&= \tau e^{ \tau L} \Pi_N F( \Pi_N U(t_n)) \\
&\quad+ (2L)^{-1} \left[
\tau e^{ \tau L} - (2L)^{-1} (e^{ \tau L} - e^{- \tau L} )\right]
\begin{pmatrix}
- \Pi_N f( \Pi_N u(t_n))\\
 \Pi_N(f^\prime( \Pi_Nu(t_n)) \Pi_N v(t_n))
\end{pmatrix}
+ R_4(t_n),
\end{align*}
where
\begin{align}\label{R3*}
R_4(t_n) = \int_0^ \tau e^{ (\tau - 2 s)L} ( \tau - s)\int_0^ s \dfrac{d}{d \sigma} \left[ e^{2 \sigma L} 
\dfrac{d}{d \sigma} e^{- \sigma L} \Pi_N F(\Pi_N \tilde U(t_n+  \sigma))
\right] d \sigma ds.
\end{align}

For $I_2(t_n)$, by approximating $\Pi_N U(t_n+\sigma)$ with $\Pi_N e^{ \sigma L} U(t_n)$, we obtain
\begin{align}\label{i1}
I_2(t_n) 
&= \int_0^ \tau e^{( \tau- s) L} \Pi_N \left[ F^\prime(\Pi_N \tilde U(t_n+s))
	\int_0^ s e^{(s- \sigma) L} \Pi_N F(\Pi_N U(t_n+ \sigma)) d \sigma\right] ds\\
	\label{i2}
&= \int_0^ \tau s e^{\tau L} \Pi_N \left[ F^\prime(\Pi_N U(t_n)) 
\Pi_N F(\Pi_N U(t_n))\right] ds
	+ R_{5}(t_n) + R_{6}(t_n) + R_{7}(t_n),
\end{align}
where the first term is equal to 0 due to \eqref{dF} and the remaining terms are as follows
\begin{align*}
R_{5}(t_n) &= \int_0^ \tau e^{( \tau-s)L} \Pi_N \Big[ F^\prime (\Pi_N \tilde U(t_n+s))\\
&\hspace{2cm}\cdot \int_0^ s e^{(s- \sigma)L} \Pi_N \big[ F(\Pi_N U(t_n+ \sigma)) 
- F( \Pi_N \tilde U(t_n+ \sigma))\big] d \sigma\Big] ds,\\
R_{6}(t_n) &= \int_0^ \tau e^{( \tau-s)L} \Pi_N \Big[ F^\prime (\Pi_N \tilde U(t_n+s))\\
&\hspace{2cm}	\cdot \int_0^ s
	\left(
		e^{(s- \sigma)L} \Pi_N F( \Pi_N \tilde U(t_n+ \sigma)) - e^{sL} \Pi_N F(\Pi_N U(t_n))
	\right)
	d \sigma \Big] ds,\\
R_{7}(t_n) &= \int_0^ \tau s e^{\tau L} \Pi_N
	\Big(e^{-s L}[F^\prime(\Pi_N \tilde U(t_n+s))
		\cdot e^{sL} \Pi_N F( \Pi_N U(t_n))]\\
&\hspace{2cm}- F^\prime( \Pi_N U(t_n)) \Pi_N F( \Pi_N U(t_n))\Big)ds.
\end{align*}
Finally, replacing $ \Pi_N F$ by $ I_N F$, we obtain
\begin{align}\label{un}
    U(t_{n+1}) &= e^{ \tau L} U(t_n) + \tau e^{ \tau L} I_N F( \Pi_N U(t_n)) \notag\\
&\quad+ (2L)^{-1} \left[
\tau e^{ \tau L} - (2L)^{-1} (e^{ \tau L} - e^{- \tau L} )\right]
H( \Pi_N U(t_n))
+ \mathcal{L}^n,
\end{align}
where $\mathcal{L}^n$ is the consistency error and is given by
\begin{align}
\label{defLn}
\mathcal{L}^n &= \sum_{i=1}^9 R_i(t_n),\\
\label{defR8}
R_8(t_n)&= \tau e^{ \tau L} ( \Pi_N - I_N) F( \Pi_N U(t_n)),\\
\label{defR9}
R_9(t_n)&= (2L)^{-1} \Big[
	\tau e^{ \tau L} - (2L)^{-1} (e^{ \tau L} - e^{- \tau L} )\Big]
   	\begin{pmatrix}
	- (\Pi_N-I_N) f( \Pi_N u(t_n))\\
	 \Pi_N((I-I_N)f^\prime( \Pi_Nu(t_n)) \Pi_N v(t_n))
	\end{pmatrix}.
\end{align}
By dropping the remainder $\mathcal{L}^n$ and taking initial value $U_{N}^{0}=\Pi_{N^{\alpha}}U(0)$, we obtain the following numerical scheme:
\begin{align}\label{modified_scheme_1}
   U^{n+1}_{N} &= e^{ \tau L} U^n_{N} + \tau e^{ \tau L} I_N F( \Pi_N U^n_{N}) + (2L)^{-1} \left[\tau e^{ \tau L} - (2L)^{-1} (e^{ \tau L} - e^{- \tau L} )H(\Pi_N U^n_{N})\right].
\end{align}
We recall the expression of $H(\Pi_N U^n_{N})$ in \eqref{defH} and apply the equality $\Pi_N I_N=I_N$. This yields the following algorithm for the low- and high-frequency parts of the numerical solution, respectively: 
\begin{align} 
\label{low_frequency}
\Pi_N U^{n+1}_{N}=\, &e^{ \tau L} \Pi_N U^n_{N}+ \tau e^{ \tau L} I_N F( \Pi_N U^n_{N})\notag \\
&\quad+ (2L)^{-1} \left[
\tau e^{ \tau L} - (2L)^{-1} (e^{ \tau L} - e^{- \tau L} )
H(\Pi_N U^n_{N})
\right], \\[5pt]
\label{high_frequency}
\Pi_{(N,N^{\alpha}]} U^{n+1}_{N}=\, &e^{ \tau L} \Pi_{(N,N^{\alpha}]} U^n_{N}.
\end{align}
This is algorithm \eqref{reformulated_scheme}, where \eqref{HR_LRI_3} is obtained by iterating \eqref{high_frequency} with respect to $n$. 

The expressions of the remainders $R_j(t_n)$, $j=1,\dots,9$, in this section are used in the error analysis in the next section.

\section{Proof of Theorem \ref{thm:convergence}  }\label{section:proof}

Let $E_N^n= U\left(t_n\right)-U_N^n$ denote the error of the numerical solution and consider the difference between \eqref{un} and \eqref{modified_scheme_1}. This yields the following error equation:
\begin{align}\label{diff}
E_N^{n+1}= &\ {e}^{\tau L} E_N^n+\tau{e}^{\tau L} I_N\left(F\left(\Pi_N U\left(t_n\right)\right)-F\left(\Pi_N U_N^n\right)\right) \notag\\
& +(2 L)^{-1}\left[\tau {e}^{\tau L}-(2 L)^{-1}\left({e}^{\tau L}- {e}^{-\tau L}\right)\right] 
\left(H (\Pi_NU(t_n))-H(\Pi_N U_N^n))\right.
+ \mathcal{L}^n,
\end{align}
with $E^0_N=(I-\Pi_{N^{\alpha}})U^{0}$. The remainder $\mathcal{L}^n$ in the expression above is estimated below.  

\begin{proposition}\label{lem:error}
Under the conditions of Theorem  \ref{thm:convergence}, the remainder $\mathcal{L}^n$ in \eqref{diff} satisfies the following estimates: 
\begin{align}\label{Ln0}
\big\|\mathcal{L}^n\big\|_{0}\lesssim \left\{ \begin{array}{ll}
\tau N^{- \gamma+} + \tau^2 N^{1-2 \gamma+} \quad &\text{for} \quad \gamma\in(0,\frac12] \\[2mm]
\tau N^{- \gamma+}+ \tau^2 N^{1-2 \gamma+}+\min(\tau^{2}, \tau^3 N^{2(1- \gamma)+}) \quad &\text{for}\quad\gamma\in(\frac12,1] 
\end{array}
\right.
\end{align} 
and
\begin{align}
\label{Ln_gamma}
&\big\|\mathcal{L}^n\big\|_{\gamma}\lesssim 
	\tau N^{- \gamma+}+  \tau^2 N^{1- \gamma+} ,
\end{align}
where the constant is independent of $\tau$.
\end{proposition}  
\begin{proof}
In view of the definition $\mathcal{L}^n = \sum_{i=1}^9 R_i(t_n)$, we present estimates for $R_j(t_n)$, $j=1,\dots,9$, respectively. 

{\sc Estimate of $R_1(t_n)$.} We decompose $R_1(t_n)$ into the following parts: 
\begin{align}
R_1(t_n)=
\label{eq:1}
& \int_0^\tau e^{(\tau-s) L} {\Pi_{>N}} F (\Pi_N U\left(t_n+s\right) ) d s \\
\label{eq:2}
& +\int_0^\tau e^{(\tau-s) L}
\left[F\left(U\left(t_n+s\right)\right)-F\left(\Pi_N U\left(t_n+s\right)\right)\right] d s.
\end{align}
Lemma \ref{bound} (Bernstein's inequality) can be used to show that 
\begin{align*}
\| \eqref{eq:1}\|_{1} \lesssim \tau N^{- \gamma}
\max_{s\in[0,\tau]}\| F( \Pi_N U(t_n+ s))\|_{1+ \gamma}\lesssim \tau N^{- \gamma}  \max_{s\in[0,\tau]} \| U(t_n+s)\|_{ \gamma}.
\end{align*}
By the mean-value theorem, we can rewrite \eqref{eq:2} into the following expression:  
\begin{align*}
\eqref{eq:2} = \int_0^\tau e^{(\tau- s)L}
 		\int_0^1 \left[ F^\prime \Big( \theta U(t_n+ s) + (1- \theta) \Pi_N U(t_n+ s)\Big) \cdot
		\Pi_{>N} U(t_n+ s)\right] d \theta d s.
\end{align*}
Then, according to \eqref{dF} and Lemma~\ref{bound}, we obtain
\begin{align*}
\| \eqref{eq:2}\|_{1}&\lesssim \tau\max_{s\in[0,\tau], \theta\in [0,1]} \| F^\prime \left( \theta U(t_n+ s) + (1- \theta) \Pi_N U(t_n + s)\right)\cdot \Pi_{>N} U(t_n+ s)\|_1\\
	&=  \tau \max_{s\in[0,\tau], \theta\in [0,1]}\| f^\prime ( \theta u(t_n+ s) + (1- \theta) \Pi_N u(t_n+ s)) \cdot \Pi_{>N} u(t_n+ s)\|_{L^2}\\
&\lesssim  \tau \max_{s\in[0,\tau]}\| \Pi_{>N} u(t_n+ s)\|_{L^2}
\lesssim \tau N^{- \gamma} \max_{s\in[0,\tau]}\| U(t_n+s)\|_ \gamma.
\end{align*}
Therefore, by collecting the estimates  \eqref{eq:1} and \eqref{eq:2}, we have the following estimate of $\|R_1(t_n)\|_1$: 
\begin{align}\label{est:R1}
\|R_1(t_n)\|_1\lesssim \tau N^{- \gamma} \max_{s\in[0,\tau]}\| U(t_n+s)\|_ \gamma.
\end{align}

{\sc Estimate of $R_2(t_n)$:}  
From \eqref{dF} we see that $\|F^\prime(U)W\|_1 \lesssim \|W\|_0$. By using this result, from the definition of $R_2(t_n)$ in \eqref{R2R3} and the definition of $\tilde R_2(s)$ in \eqref{def-tilde-R2} we derive that 
\begin{align}\label{est:R2}
\|R_2(t_n)\|_1&\lesssim \tau \max_{s\in[0,\tau]}\|\tilde R_2(s)\|_1 
\lesssim \tau^2\max_{\sigma\in[0,\tau]}\big\|F( U(t_n+ \sigma)-F(\Pi_N U(t_n+ \sigma))\big\|_0 \notag\\
& \lesssim \tau^2 \max_{\sigma\in[0,\tau]}\|f(u(t_n+\sigma))-f(\Pi_N u(t_n+ \sigma))\|_{H^{-1}} \notag\\
& \lesssim \tau^2 \max_{\sigma\in[0,\tau]}\|f(u(t_n+\sigma))-f(\Pi_N u(t_n+ \sigma))\|_{L^2} \notag\\
&\lesssim \tau^2 N^{-\gamma} \max_{s\in[0,\tau]}\| U(t_n+s)\|_ \gamma , 
\end{align}
where the last inequality follows from the Lipschitz continuity of $f$ and Lemma \ref{bound}. 

{\sc Estimate of $R_3(t_n)$:}  
Using the definition of $R_F(s)$ below \eqref{def-tilde-R3}, and the expression of $F''(U)$ in \eqref{dF}, we derive  that 
\begin{align*}
\| R_F(s) W \cdot W^*\|_1 \lesssim \|f^{\prime\prime} \|_{L^\infty} \|w_1w_1^*\|_{L^2}
\lesssim \| W\|_{1} \| W^*\|_1.
\end{align*}
Then, in view of the definition of $R_3(t_n)$ in \eqref{R2R3}, we have
\begin{align}\label{est:R3}
 \| R_3(t_n)\|_1 
& \lesssim\tau \max_{s\in[0,\tau]} \Big| \int_0^s  \| \Pi_N F(U(t_n+ \sigma))\|_1 d \sigma\Big|^2
\lesssim \tau^3 ,
\end{align}
where the last inequality follows from that 
$$ \| F(U(t_n+ \sigma))\|_1 = \| f(u(t_n+ \sigma))\|_0 \lesssim \|f(0)\|_0 + \|f'\|_{L^\infty} \|u(t_n+ \sigma)\|_{L^2} . $$

{\sc Estimate of $R_4(t_n)$:} 
We rewrite $R_4(t_n)$ as 
\begin{align}\label{rewrite_R4}
R_4(t_n) = \int_0^ \tau e^{ (\tau - 2 s)L} ( \tau - s)\int_0^ s \dfrac{d}{d \sigma} \left[ e^{2 \sigma L} G^\prime(t_n+ \sigma)\right] d \sigma ds 
\end{align}
with
$$
G(t_n+ s) = e^{- s L} \Pi_N F(\Pi_N \tilde U(t_n+ s)) . 
$$ 
By applying Lemma~\ref{mainlem} with $U = \Pi_N U(t_n)$ and taking advantage of the commutative property bewteen $ \Pi_N$ and $L$, we obtain 
\begin{align}\label{dGs}
&G^\prime (t_n+ s)
= e^{- s L} \begin{pmatrix} - \Pi_N f( \Pi_N \tilde u(t_n + s)) \\
 \Pi_N \left( f^\prime( \Pi_N\tilde u(t_n+ s)) \Pi_N \tilde v(t_n+ s)\right) \end{pmatrix},\\[2mm]
\label{ddGs}
&\frac{d}{ds} \left[e^{2sL}G^\prime (t_n +s)\right]  
=e^{sL} \begin{pmatrix} 0\\
\Pi_N \mathcal{F}(\Pi_N U)(t_n+s)
\end{pmatrix},
\end{align}
with 
\begin{align*}
&\mathcal{F}(\Pi_N U)(t_n+s) \\
&=f^{\prime \prime}(\Pi_N \u)
\left( | \Pi_N \v|^2 - |\Pi_N \nabla \u|^2\right)+m (\Pi_N \u-f(\Pi_N \u)).
\end{align*}

On the one hand, we can estimate $\| R_4(t_n)\|_0 $ by using its expression in \eqref{rewrite_R4} and the expression of $\frac{d}{ds} \left[e^{2sL}G^\prime (t_n +s)\right]$ in \eqref{ddGs}. Then, using the negative-norm estimates in \eqref{embedding2} and Bernstein's inequality in Lemma~\ref{bound}, we have 
\begin{align}\label{0_norm_R4}
\| R_4(t_n)\|_0 &\lesssim \tau^{3}\max_{s\in[0,\tau]} \|\mathcal{F}(\Pi_N U)(t_n+s)\|_{H^{-1}} \notag\\
&\lesssim \tau^3 \max_{s\in[0,\tau]}\| f^{\prime\prime}( \Pi_N \u) \left( | \Pi_N \v|^2- |\Pi_N \nabla \u |^2\right)\|_{H^{-1}} \notag\\
&\quad+\tau^{3} \max_{s\in[0,\tau]}\|\Pi_N \u-f(\Pi_N \u)\|_{H^{-1}}\notag\\
&\lesssim \tau^3 \max_{s\in[0,\tau]}\left( \| \Pi_N \v\|_{L^{2+}}^2 + \| \Pi_N \nabla \u\|_{L^{2+}}^2+\|\u\|_{L^{2}}+1\right) \notag\\
&\lesssim \tau^3 N^{2(1- \gamma)+} \| U(t_n)\|_{ \gamma}^2.
\end{align}
On the other hand, we can also estimate $R_4(t_n)$ by using its following reformulated expression:
\begin{align}\label{defR42}
R_4(t_n) = \int_0^ \tau e^{ (\tau - 2 s)L} ( \tau - s) \left[ e^{2 s L} G^\prime(t_n+s)-G^{\prime}(t_n)\right] ds.
\end{align}
In this way, using the definition of $G^{\prime}$ in \eqref{dGs}, we have 
\begin{align*}
    \|R_4(t_n) \|_0 &\lesssim \tau^2 \sup_{s\in[0,\tau]}
    \|G^\prime(t_n+ s)\|_0\\
&\lesssim \tau^2 \sup_{s\in[0,\tau]} 
\big(
 \|\Pi_N f( \Pi_N \tilde u(t_n+s))\|_{L^2}
  +\|\Pi_N \left( f^\prime( \Pi_N\tilde u(t_n+ s)) \Pi_N \tilde v(t_n+ s)\right) \|_{H^{-1}}
\big).
\end{align*}
According to \eqref{propL}, $\tilde u$ is bounded in $L^2$ and therefore the first term on the right-hand side above is $O(\tau^2)$. By employing Lemma \ref{lem:H_estimate}, we can estimate the second term on the right-hand side above as follows: 
\begin{align*}
\tau^2 \|\Pi_N \left( f^\prime( \Pi_N\tilde u(t_n+ s)) \Pi_N \tilde v(t_n+ s)\right) \|_{H^{-1}} 
\lesssim 
\left\{\begin{aligned}
\tau^2 N^{1-2\gamma+}&\| U(t_n)\|_{ \gamma}^2 &&\mbox{for}\,\,\, \gamma\in\mbox{$(0,\frac12]$} \\[1mm] 
\tau^2 &\| U(t_n)\|_{ \gamma}^2  &&\mbox{for}\,\,\, \gamma\in\mbox{$(\frac12,1]$} . 
\end{aligned}\right.
\end{align*}
This, together with \eqref{0_norm_R4}, yields the following estimate of $R_4(t_n)$ in the $0$-norm: 
\begin{align}\label{0_norm_R4_new}
\|R_4(t_n)\|_{0}\lesssim 
\left\{\begin{aligned}
\tau^2 N^{1-2 \gamma+} &\| U(t_n)\|_{ \gamma}^2 &&\mbox{for}\,\,\, \gamma\in\mbox{$(0,\frac12]$} \\[1mm]
\min(\tau^{2}, \tau^3 N^{2(1- \gamma)+})\cdot &\| U(t_n)\|_{ \gamma}^2 &&\mbox{for}\,\,\, \gamma\in\mbox{$(\frac12,1]$}  .
\end{aligned}\right.
\end{align}

By considering the $1$-norm of right-hand side of \eqref{defR42}, and using the expression of $G'(t_n+s)$ in \eqref{dGs}, we obtain the following estimate of $R_4(t_n)$ in the $1$-norm: 
\begin{align}\label{1_norm_R4}
\| R_4(t_n)\|_ 1 &\lesssim \tau^{2} \max_{s\in[0,\tau]}\left(\|\Pi_N f( \Pi_N \tilde u(t_n + s))\|_{H^{1}}+\|\Pi_N \left( f^\prime( \Pi_N\tilde u(t_n+ s)) \Pi_N \tilde v(t_n+ s)\right)\|_{L^{2}}\right)\notag\\
&\lesssim \tau^{2}\max_{s\in[0,\tau]}\Big(\|f( \Pi_N \tilde u(t_n + s))\|_{L^{2}}+\|f^{\prime}( \Pi_N \tilde u(t_n + s))\nabla \Pi_N \tilde u(t_n + s)\|_{L^{2}}\notag\\
&\qquad\qquad\qquad\qquad\qquad\qquad\qquad\qquad+\|f^\prime( \Pi_N\tilde u(t_n+ s)) \Pi_N \tilde v(t_n+ s)\|_{L^{2}}\Big)\notag\\
&\lesssim \tau^{2}N^{1-\gamma}\| U(t_n)\|_{ \gamma}.
\end{align}

{\sc Estimate of $R_5(t_n)$:}
Using \eqref{propL}, \eqref{dF} and \eqref{duhamel}, we obtain
\begin{align}\label{1_norm_R5}
\notag
\| R_{5}(t_n)\|_1 &\lesssim \int_0^ \tau \int_0^ s 
\| F(\Pi_N U(t_n+ \sigma)) - F(\Pi_N \tilde U(t_n+ \sigma))\|_0
d \sigma ds,\\
\notag
&\lesssim \tau^2 \max_{ \sigma\in [0, \tau]} \| U(t_n+ \sigma) - e^{ \sigma L} U(t_n)\|_0 \\
&\lesssim \tau^3 (1+\max_{ \sigma\in [0,\tau]} \| U(t_n+ \sigma)\|_0).
\end{align}
Further, we rewrite $R_{6}(t_n) $ as
\begin{align*}
R_{6}(t_n) = \int_0^ \tau e^{( \tau - s) L} \Pi_N \left[
F^\prime(e^{sL} \Pi_NU(t_n))\cdot
\int_0^ s e^{sL} \int_0^ \sigma 
G'(t_n+ \rho) d \rho d \sigma
\right] ds.
\end{align*}
Taking the 1 norm and using \eqref{propL} and \eqref{dF}, we obtain
\begin{align*}
\| R_{6}(t_n)\|_1 &\lesssim \int_0^ \tau \left\|
F^\prime(e^{sL} \Pi_NU(t_n))\cdot
\int_0^ s e^{sL} \int_0^ \sigma 
G'(t_n+ \rho) d \rho d \sigma
\right\|_1 ds\\
&\lesssim \int_0^ \tau 
\left\|\int_0^ s e^{sL} \int_0^ \sigma 
G'(t_n+ \rho) d \rho d \sigma
\right\|_0 ds\\
&\lesssim \int_0^ \tau \int_0^ s \int_0^ \sigma
	\big\|  G'(t_n+ \rho)
	\big\|_0 d \rho d \sigma ds.
 \end{align*}
Substituting \eqref{dGs} into above, we have
 \begin{align}\label{1_norm_R6}
\| R_{6}(t_n)\|_1 
&\lesssim \tau^3 \left(1 + \|\Pi_N u(t_n)\|_{L^2} + \|  \Pi_N   v(t_n)\|_{L^2}\right)
\lesssim \tau^3 N^{1- \gamma} \left( 1+ \|  U(t_n)\|_ \gamma\right).
\end{align}

 Similar, we rewrite $R_{7}(t_n)$ as
 \begin{align*}
 R_{7}(t_n) 
 = \int_0^ \tau s e^{\tau L} \int_0^ s \dfrac{d}{d \sigma} e^{- \sigma L} \Pi_N \left[
 F^\prime(e^{\sigma L} \Pi_N U(t_n)) e^{ \sigma L} \Pi_N F( \Pi_N U(t_n))
 \right] d \sigma ds.
 \end{align*}
Using \eqref{inequdF} with $U = \Pi_N U(t_n)$ and $W = \Pi_N F( \Pi_N U(t_n))$, from \eqref{propL}, we obtain
\begin{align}\label{deF}\notag
&\Big\| \dfrac{d}{d \sigma} e^{- \sigma L} 
\left[F^\prime(e^{\sigma L} \Pi_N U(t_n)) e^{ \sigma L} \Pi_N F( \Pi_N U(t_n)) \right]\Big\|_1\\
\notag
&\lesssim \|\Pi_N F( \Pi_N U(t_n))\|_{\frac32+ \epsilon} \|  \Pi_N \tilde U(t_n+ \sigma)\|_1 
+ \| \Pi_N F( \Pi_N U(t_n))\|_1\\
\notag&\lesssim \|\Pi_N f( \Pi_N u(t_n))\|_{H^{\frac12+ \epsilon}} \|  \Pi_N  U(t_n)\|_1 
+ \| \Pi_N f( \Pi_N u(t_n))\|_0\\
&\lesssim \| \Pi_N U(t_n)\|_1^2+ \| \Pi_NU(t_n)\|_1+ 1.
\end{align}
Taking 1 norm of $R_{7}(t_n) $ and using \eqref{deF}, we obtain
\begin{align}\label{1_norm_R7}
\| R_{7}(t_n)\|_1 &\lesssim \tau^3 \left(\| \Pi_N U(t_n)\|_1^2 + \| \Pi_N U(t_n)\|_1 + 1 \right)\lesssim  \tau^3 N^{2(1- \gamma)} \| U(t_n)\|_ \gamma^2.
\end{align}
Putting together the estimates  \eqref{1_norm_R5},  \eqref{1_norm_R6} and \eqref{1_norm_R7} yields
\begin{align}\label{es:i2}
    \|I_2(t_n)\|_1\lesssim \tau^3 N^{2(1- \gamma)}.
\end{align}

Next, we present another bound of $ \| I_2(t_n)\|_1$.
To this end, we apply the 1 norm to both sides of \eqref{i1}, and use \eqref{propL} and \eqref{dF} to derive
\begin{align}\label{rewrite_I2}\notag
\|I_2(t_n) \|_1
&\lesssim \tau^2 \max_{s,\sigma\in[0,\tau]}
 \left\| F^\prime(\Pi_N \tilde U(t_n+s))
  e^{(s- \sigma) L} \Pi_N F(\Pi_N U(t_n+ \sigma)) \right\|_1\\
  \notag
&\lesssim \tau^2 \max_{s,\sigma\in[0,\tau]} 
\|f^{\prime}(\Pi_{N}\tilde{u}(t_{n}+s))\|_{L^\infty} 
\|F(\Pi_N U(t_n+ \sigma))\|_{0}\\
&\lesssim \tau^2 (1+ \| U(t_n)\|_0).
\end{align}

Inequalities \eqref{es:i2} and \eqref{rewrite_I2} yields
\begin{align}\label{1_norm_I2}
\|R_{5}(t_n)+R_{6}(t_n)+R_{7}(t_n)\|_1=\|I_2(t_n) \|_1\lesssim \min(\tau^3 N^{2(1- \gamma)}, \tau^2),
\end{align}
for all $0<\gamma\leq 1$.

{\bf Estimate of $R_8$ and $R_9$: } By using Lemma~\ref{bound} and Lemma~\ref{lemInterp}, we obtain
\begin{align}\label{est:R8}
\| R_8\|_1
	&\lesssim \tau \| (\Pi_N-I_N) f( \Pi_N u(t_n))\|_{L^2}
	\lesssim \tau N^{-1-} \| f( \Pi_N u(t_n))\|_{H^{1+}}
	\lesssim \tau N^{- \gamma+}.
\end{align}
For $\|R_9\|_0$, by using \eqref{embedding2}, Lemma~\ref{bound} and Lemma~\ref{lemInterp}, we obtain
\begin{align}\label{0_norm_R9}
\| R_9\|_0 &\lesssim \tau^2 \left( \| (\Pi_N - I_N) f( \Pi_N u(t_n))\|_{L^2} +
	\|(I-I_N)f^\prime( \Pi_N u(t_n))\cdot \Pi_N v(t_n)\|_{H^{-1}}\right)\notag\\
	& \lesssim \tau^2 \left( N^{-1-} \| f( \Pi_N u(t_n))\|_{H^{1+}}
	+ \| (I-I_N)f^\prime( \Pi_N u(t_n))\|_{L^2} \|  \Pi_N v(t_n)\|_{L^{2+}}\right)\notag\\
	&\lesssim \tau^2 N^{- \gamma+} + \tau^2 N^{1-2 \gamma+}.
\end{align}
For $ \| R_9\|_ \gamma$, according to \eqref{embedding}, we derive
\begin{align}\label{R6d2}
 \| R_9\|_ \gamma &\lesssim \tau^2 \left(
 	\| ( \Pi_N - I_N) f( \Pi_N u(t_n))\|_{H^ \gamma} + 
	\|(I-I_N)f^\prime( \Pi_N u(t_n)) \Pi_n v(t_n)\|_{H^{ \gamma-1}}
 	\right)\notag\\
	&\lesssim \tau^2 N^{-1+\gamma-} \| f( \Pi_N u(t_n))\|_{H^{1+}}+ \tau^2 \| (I-I_N)f^\prime( \Pi_N u(t_n))\|_{L^2}
		\| \Pi_N v(t_n)\|_{H^{ \gamma+}}\notag\\			
	&\lesssim \tau^2 N^{0+}+ \tau^2 N^{1 - \gamma+} \| v(t_n)\|_{H^ {\gamma-1}}.
\end{align}
By summing up \eqref{est:R1}-\eqref{est:R3},
\eqref{0_norm_R4_new}, \eqref{1_norm_I2}--\eqref{0_norm_R9}, 
we get the estimate \eqref{Ln0} in the proposition~\ref{lem:error}. 
And combining \eqref{est:R1}-\eqref{est:R3}, \eqref{1_norm_R4}, \eqref{1_norm_I2}--\eqref{est:R8} and \eqref{R6d2} we finish the proof of \eqref{Ln_gamma} in the proposition~\ref{lem:error}.
\end{proof}

\begin{lemma}\label{lem1}
Let $0<\gamma \leq1$.
There exists $ \tau_0\in (0,1)$ such that for $ \tau\in (0, \tau_0)$,
$N\le \tau^{-\frac{1}{1- \gamma}+\epsilon_{0}}$, for some $0<\epsilon_{0}\ll 1$, we have the following estimate
$$
 \big\|U^n_N\big\|_{\gamma}\le C,
$$
 where the constant  $C$ depends only on  $\|U(t_n)\|_{\gamma}$, $\|f^{\prime}\|_{L^{\infty}}$ and $\|f^{\prime\prime}\|_{L^{\infty}}$.
\end{lemma}
\begin{proof}
To apply Gronwall's inequality, we iterate \eqref{diff} with respect to $n$. This yields the following expression: 
\begin{align}\label{sumid}
    E^{n+1}_N=e^{(n+1)\tau L} E^0_N
    +\sum_{j=0}^n e^{(n-j)\tau L} \mathcal{R}^j,
\end{align}
where $\mathcal{R}^n$ is given as follows:
\begin{align}\label{def-R^n}
    \mathcal{R}^n&=
    \tau {e}^{\tau L} I_N\big( F(\Pi_N U(t_n))
    		-F(\Pi_N U_N^n) \big) \notag\\
& \quad +(2 L)^{-1}
	\big[\tau {e}^{\tau L}
		-(2 L)^{-1}\big({e}^{\tau L}
		- {e}^{-\tau L}\big)\big] 
	\big(H (\Pi_NU (t_n) )
	-H (\Pi_N U_N^n)\big)
	+  \mathcal{L}^n.
\end{align}

By considering the $\gamma$-norm of both sides of \eqref{sumid}, we derive that
\begin{align}\label{n1}
\| E_N^{n+1}\|_ \gamma & \leq
C_0\Big(\| E_N^{0}\|_ \gamma+\sum_{j=0}^n\|\mathcal{R}^j\|_\gamma\Big).
\end{align}
The second term in the expression of $\mathcal{R}^n$ can be estimated by noticing and utilizing the following relation: 
\begin{align}\label{identiy}
&\quad\int_{0}^{s}{e}^{(\tau-2s) L}(\tau-s)
\big(H (\Pi_NU (t_n) )-H (\Pi_N U_N^n)\big)ds \\
&=(2 L)^{-1}\left[\tau {e}^{\tau L}-(2 L)^{-1}\left({e}^{\tau L}- {e}^{-\tau L}\right)\right] 
\big(H (\Pi_NU (t_n) )-H (\Pi_N U_N^n)\big). \notag
\end{align}
Using the definition of $H$ in \eqref{defH}, we have
\begin{align}\label{Rn}
\|\mathcal{R}^n\|_\gamma&	
	\lesssim \tau \| I_{N}\left[F( \Pi_N U(t_n))- F( \Pi_N U_N^n)\right]\|_ \gamma 
	+ \tau^2 \| I_N\left[f( \Pi_N u(t_n) ) - f( \Pi_N u_N^n)\right]\|_{H^\gamma} \notag \\
	&\quad+ \tau^2 \| I_N f^\prime (\Pi_N u(t_n)) \cdot \Pi_N v(t_n)
			- I_N f^\prime (\Pi_N u_N^n) \cdot \Pi_N v_N^n\|_{H^{\gamma-1}}
		+ \| \mathcal{L}^n\|_ \gamma \notag\\
&=: J^n_1+J^n_2+J^n_3+\| \mathcal{L}^n\|_ \gamma.
\end{align}
Regarding the presence of $I_N$ which may be unbounded operator in $H^{\gamma-1}$, we decompose $J^{n}_{1}$ as
\begin{align}\label{Rn1}
J^n_1 &\lesssim \tau\| f( \Pi_N u(t_n))- f( \Pi_N u_N^n)\|_ {H^{\gamma-1}}+\tau\| (I-I_N)\left[f( \Pi_N u(t_n))- f( \Pi_N u_N^n)\right]\|_ {H^{\gamma-1}}\notag\\
&=: J^{n}_{11}+J^{n}_{12}.
\end{align}
For the first term in \eqref{Rn1}, we get
\begin{align}\label{Rn11}
J^{n}_{11}\leq\tau\left\| f( \Pi_N u(t_n))- f( \Pi_N u_N^n)\right\|_{L^{2}}\lesssim \tau \|f'\|_{L^\infty} \| \Pi_N E^n_N\|_\gamma.
\end{align}

For $J^{n}_{12}$ in \eqref{Rn1}, we apply the error estimate of the interpolation operator $I_N$ in Lemma~\ref{lemInterp} to prove
\begin{align}\label{Rn12}
J^{n}_{12}&\leq \tau\| (I-I_N)\left[f( \Pi_N u(t_n))- f( \Pi_N u_N^n)\right]\|_ {L^2}\notag \\
&\lesssim \tau N^{-1-}\|f(\Pi_N u(t_n))-f(\Pi_N u^{n}_{N})\|_{H^{1+}}\notag\\
&\lesssim \tau N^{-\gamma+}(\|\Pi_N u(t_n)\|_{H^{\gamma}}+\|\Pi_N E^n_N\|_\gamma)^{1+},
\end{align}
where the third inequality follows from 
\begin{align}\label{3rd_Ineq_Rn12}
&\|f(\Pi_N u(t_n))-f(\Pi_N u^{n}_{N})\|_{H^{1+\varepsilon}}\notag\\
&\quad\lesssim \|f(\Pi_N u(t_n))-f(\Pi_N u^{n}_{N})\|^{1-\varepsilon}_{H^{1}}\|f(\Pi_N u(t_n))-f(\Pi_N u^{n}_{N})\|^{\varepsilon}_{H^{2}}\notag\\
&\quad\lesssim \left(\|\Pi_N u(t_n)\|_{H^{1}}+\|\Pi_N u_N^{n}\|_{H^{1}}\right)^{1-\varepsilon}\notag\\
&\quad\quad\cdot\left(\|\Pi_N u(t_n)\|_{H^{2}}+\|\Pi_N u(t_n)\|^{2}_{H^{1+\frac{d}{4}}}+\|\Pi_N u^{n}_{N}\|_{H^{2}}+\|\Pi_N u^{n}_{N}\|^{2}_{H^{1+\frac{d}{4}}}\right)^{\varepsilon}\notag\\
&\quad \lesssim (N^{1-\gamma})^{1-\varepsilon}\left(\|\Pi_N u(t_n)\|_{H^{1}}+\|\Pi_N u_N^{n}\|_{H^{1}}\right)^{1-\varepsilon}\notag\\
&\quad\quad \cdot (N^{2-\gamma}+N^{2(1+\frac{d}{4}-\gamma)})^{\varepsilon}\left(\|\Pi_N u(t_n)\|_{H^{\gamma}}+\|\Pi_N u(t_n)\|^{2}_{H^{\gamma}}+\|\Pi_N u^{n}_{N}\|_{H^{\gamma}}+\|\Pi_N u^{n}_{N}\|^{2}_{H^{\gamma}}\right)^{\varepsilon}\notag\\
&\quad\lesssim N^{1-\gamma+\varepsilon^{*}(\gamma)}(\|\Pi_N u(t_n)\|_{H^{\gamma}}+\|\Pi_N E^n_N\|_{\gamma})^{1+\varepsilon},
\end{align} 
with
\begin{align*}
\varepsilon^{*}(\gamma)=\left\{ \begin{array}{ll}
\varepsilon \quad &\text{for} \quad \gamma\geq\frac{d}{2}, \\[2mm]
(1-\gamma+\frac{d}{2})\varepsilon \quad &\text{for}\quad\gamma<\frac{d}{2}.
\end{array}
\right.
\end{align*}

For the term $J^{n}_2$, we decompose it as
\begin{align}\label{Rn2}
J^n_2 &\lesssim \tau^{2}\| f( \Pi_N u(t_n))- f( \Pi_N u_N^n)\|_ {H^{\gamma}}+\tau^{2}\| (I-I_N)\left[f( \Pi_N u(t_n))- f( \Pi_N u_N^n)\right]\|_ {H^{\gamma}}\notag\\
&=: J^{n}_{21}+J^{n}_{22}.
\end{align}
By the similar argument as that of \eqref{Rn11} and \eqref{Rn12}, we have
\begin{align}\label{Rn21}
J_{21}^n&\leq \tau^{2}\| f( \Pi_N u(t_n))- f( \Pi_N u_N^n)\|_ {H^{1}}\notag\\
&\lesssim \tau^2 \| \Pi_N (u(t_n) - u_N^n)\|_{L^2} +
\tau^2\|f^\prime(\Pi_N u(t_n)) \Pi_N \nabla u(t_n)- f^\prime( \Pi_N u_N^n) \Pi_N \nabla u_N^n\|_{L^2}\notag\\\notag
&\lesssim \tau^2\|\Pi_N E_N^n\|_{ \gamma} +
\tau^2\|f^{\prime}\|_{L^{\infty}} \|\Pi_N \nabla u(t_n)\|_{L^2} 
+ \tau^2\|f^{\prime}\|_{L^{\infty}}\|\Pi_N \nabla(u(t_n)- u_N^n)\|_{L^2}\\
&\lesssim \tau^2N^{1- \gamma}(1+\|\Pi_N E_N^n\|_{ \gamma}).
\end{align}
and
\begin{align}\label{Rn22}
J^{n}_{22}&\leq \tau^{2}\| (I-I_N)\left[f( \Pi_N u(t_n))- f( \Pi_N u_N^n)\right]\|_ {H^{\gamma}}\notag\\
&\lesssim \tau^{2} N^{-(1-\gamma)-}\|f(\Pi_N u(t_n))-f(\Pi_N u^{n}_{N})\|_{H^{1+}}\notag\\
&\lesssim \tau^{2} N^{0+}(\|\Pi_N u(t_n)\|_{H^{\gamma}}+\|\Pi_N E^{n}_{N}\|_{\gamma})^{1+}.
\end{align}
where we have used \eqref{3rd_Ineq_Rn12} in the last inequality.

For the third term of \eqref{Rn}, we have

\begin{align}\label{Rn3}
J_3^{n}
&\leq\tau^2\| ( f^\prime (\Pi_N u(t_n))-  f^\prime (\Pi_N u_N^n)) \Pi_N v_N^n\|_{ H^{\gamma-1}}\notag\\
&\quad+\tau^2\| (I-I_N)\left( f^\prime (\Pi_N u(t_n)) 
-  f^\prime (\Pi_N u_N^n)\right) \Pi_N v_N^n\|_{ H^{\gamma-1}}\notag\\
&\quad +\tau^2\|  f^\prime (\Pi_N u(t_n)) \Pi_N (v(t_n)-v_N^n)\|_{ H^{\gamma-1}}\notag\\
&\quad +\tau^2\|  (I-I_N) f^\prime (\Pi_N u(t_n)) \Pi_N (v(t_n)-v_N^n)\|_{ H^{\gamma-1}}\notag\\
&=:J^{n}_{31}+J^{n}_{32}+J^{n}_{33}+J^{n}_{34}
\end{align}

We first consider the case $\gamma\in(0,1)$. By the proof of \eqref{embedding}, for $d=1,2$, we have
\begin{align}\label{Rn31}\notag
J^{n}_{31}+J^{n}_{33}
&\lesssim \tau^2
\big(
\| \Pi_N v_N^n\|_{L^2}
\| f''\|_{L^\infty} \|\Pi_N u(t_n) -\Pi_N u_N^n\|_{L^{\frac{d}{1-\gamma}}}
+\|f^\prime\|_{L^\infty} \|\Pi_N v(t_n)-\Pi_N v_{N}^{n}\|_{L^2}
\big)\\
&\lesssim  \tau^{2}N^{1- \gamma}
\|\Pi_N E_N^n\|_ \gamma (1 + \|\Pi_N E_N^n\|_ \gamma).
\end{align}
By applying Lemma~\ref{lemInterp} and Sobolev embedding, we deduce
\begin{align}\label{Rn32}
J^{n}_{32}&\lesssim \tau^{2}\| \Pi_N v_N^n\|_{L^{\frac{d}{1-\gamma}}}\|(I-I_N)\left( f^\prime (\Pi_N u(t_n)) 
-  f^\prime (\Pi_N u_N^n)\right) \|_{L^{2}}\notag\\
&\lesssim \tau^{2} \| \Pi_N v_N^n\|_{H^{\gamma}}\cdot N^{-1-}\| f^\prime (\Pi_N u(t_n)) 
-  f^\prime (\Pi_N u_N^n)\|_{H^{1+}}\notag\\
&\lesssim \tau^{2} N\| \Pi_N v_N^n\|_{H^{\gamma-1}}\cdot N^{-\gamma+}\left(\|\Pi_N u(t_n)\|_{H^{\gamma}}+\|\Pi_N E_{N}^n\|_{\gamma}\right)^{1+}\notag\\
&\lesssim \tau^{2}\cdot N^{1-\gamma+}\left(\|\Pi_N v(t_n)\|_{H^{\gamma-1}}+\|\Pi_N E_{N}^n\|_{\gamma}\right)\cdot\left(\|\Pi_N u(t_n)\|_{H^{\gamma}}+\|\Pi_N E_{N}^n\|_{\gamma}\right)^{1+},
\end{align}
where we have used \eqref{3rd_Ineq_Rn12} in the third inequality.
Similarly, we can prove
\begin{align}\label{Rn34}
J^{n}_{34}\lesssim \tau^{2}N^{1-\gamma+}\|\Pi_N E^{n}_{N}\|_{\gamma}\|\Pi_N u(t_n)\|^{1+}_{H^{\gamma}}.
\end{align}

For the case $\gamma=1$, one can prove that
\begin{align}\label{Rn31_gamma=1}\notag
J^{n}_{31}+J^{n}_{33}
&\lesssim \tau^2
\big(
\| \Pi_N v_N^n\|_{L^2}
\| f''\|_{L^\infty} \|\Pi_N u(t_n) -\Pi_N u_N^n\|_{L^{\infty}}
+\|f^\prime\|_{L^\infty} \|\Pi_N v(t_n)-\Pi_N v_{N}^{n}\|_{L^2}
\big)\\
&\lesssim  \tau^{2}N^{1- \gamma+}
\|\Pi_N E_N^n\|_ \gamma (1 + \|\Pi_N E_N^n\|_ \gamma).
\end{align}
and 
\begin{align}\label{Rn32_gamma=1}
J^{n}_{32}+J^{n}_{34}\lesssim \tau^{2}N^{1-\gamma+}\left(\|\Pi_N v(t_n)\|_{H^{\gamma-1}}+\|\Pi_N E_{N}^n\|_{\gamma}\right)\cdot\left(\|\Pi_N u(t_n)\|_{H^{\gamma}}+\|\Pi_N E_{N}^n\|_{\gamma}\right)^{1+},
\end{align}
as in the same discussions of \eqref{Rn31} and \eqref{Rn32}.

Substituting \eqref{Rn1}--\eqref{Rn32_gamma=1}  and \eqref{Ln_gamma} for the estimates $J_i^{n}$, $i=1,2,3$ and $\|\mathcal{L}^{n}\|_{\gamma}$ into \eqref{Rn}, and noting that $\|U(t_n)\|_{\gamma}\lesssim 1$ and $\|\Pi_N E_{N}^n\|_{\gamma}\lesssim \|E_{N}^n\|_{\gamma}$,  we get
\begin{align}\label{estimate_Rn}
\|\mathcal{R}^n\|_\gamma&	
	\lesssim \tau \|E_{N}^n\|_{\gamma} +(\tau N^{-\gamma+}+\tau^{2}N^{1-\gamma+})\|E^{n}_N\|_{\gamma}^{2+}+\tau N^{-\gamma+}+\tau^{2}N^{1-\gamma+}.
\end{align}


Under the assumption $N\le \tau^{-\frac{1}{1- \gamma}+\epsilon_{0}}$, there exist $\epsilon=(1-\gamma)\epsilon_{0}-$ such that
\begin{align*}
\tau N^{1- \gamma+}\leq \tau^{1-(1-\gamma)(\frac{1}{1-\gamma}-\epsilon_{0})}\cdot N^{0+}= \tau^{\epsilon}.
\end{align*}

Therefore,  $\mathcal{R}^n$ satisfies 
\begin{align*}
    \|\mathcal{R}^n\|_\gamma \leq C\tau(\|E^n_N\|_\gamma+\tau^\epsilon\|E^n_N\|_\gamma^{2+}+1).
\end{align*}
Combining the above inequality with \eqref{n1} yields
\begin{align}\label{Eequ}
    \|E^{n+1}_N\|_\gamma \leq C_*(\|E^0_N\|_\gamma+1)
+C_*\tau\sum_{j=0}^n\big(\|E^j_N\|_\gamma+\tau^\epsilon\|E^j_N\|_\gamma^{2+}\big).
\end{align}
We choose $\tau$ small enough such that
\begin{align}\label{assumptau}
    \tau^\epsilon \left(e^{2C_*T}C_*(1+\|E^0_N\|_\gamma)\right)^{1+}\le1.
\end{align}
Then, we use the induction to prove 
\begin{align}\label{argun}
    \|E^n_N\|_\gamma \leq C_*(1+2C_*\tau)^n(1+\|E^0_N\|_\gamma).
\end{align}
For $n=0$, it is trivial. We assume it holds for any $n\le m$.
We shall prove it holds for $n=m+1$.
From \eqref{assumptau}, and the induction hypothesis, we obtain 
\begin{align*}
\tau^\epsilon\|E^n_N\|^{1+}_\gamma\le1,\quad n\le m.
\end{align*}
The above inequality simplifies the quadrature term in \eqref{Eequ}. Thus, we deduce
\begin{align*}
\| E_N^{m+1}\|_{\gamma} &\le C_* ( \| E_N^{0}\|_{\gamma} +1) + 2\tau C_* \sum_{j=0}^m \| E_N^{j}\|_{\gamma}\ (\textrm{by induction hypothesis})\\
&\le C_*(\| E_N^{0}\|_{\gamma}+1) + C_* (\| E_N^{0}\|_{\gamma}+1) \Big[ (1+2C_* \tau)^{m+1}-1\Big]\\
&\le C_*(\| E_N^{0}\|_{\gamma}+1)(1+2C_*\tau)^{m+1},
\end{align*}
which proves \eqref{argun} for $n = m+1$ and ends the mathematical induction. 
Furthermore, we deduce that
\begin{align*}
    \|E^{n}_N\|_\gamma\leq
    C_*e^{2C_* T}(1+\|E^0_N\|_\gamma)=C_*e^{2C_* T}\left(1+\|\Pi_{>N^{\alpha}}U^{0}\|_\gamma\right)\lesssim 1.
\end{align*}
That implies the result of the lemma.
\end{proof}

By considering the $0$-norm of $E_N^{n+1}$ using the expression in \eqref{sumid}, we have 
\begin{align}\label{eq12}
\| E_N^{n+1}\|_ 0 & \leq
C_0\Big(\| E_N^{0}\|_ 0+\sum_{j=0}^n\|\mathcal{R}^j\|_0\Big) . 
\end{align}
As in the proof of Lemma~\ref{lem1}, we recall the identity \eqref{identiy} and the definition of $H$ in \eqref{defH} to get
\begin{align}\label{estimate_Rn0}
\|\mathcal{R}^n\|_0&	\lesssim \tau \| f( \Pi_N u(t_n))- f( \Pi_N u_N^n)\|_ {H^{-1}}+\tau \| (I-I_N)\left(f( \Pi_N u(t_n))- f( \Pi_N u_N^n)\right)\|_ {H^{-1}} \notag\\
&\quad+\tau^2 \| f( \Pi_N u(t_n) ) - f( \Pi_N u_N^n)\|_{L^2}+ \tau^2 \| (I-I_N)\left( (f( \Pi_N u(t_n) ) - f( \Pi_N u_N^n))\right)\|_{L^2}\notag\\
	&\quad+ \tau^2 \|\left( f^\prime (\Pi_N u(t_n))-f^\prime (\Pi_N u^{n}_{N}) \right)\cdot \Pi_N v(t_n)\|_{H^{-1}}\notag\\
	&\quad+\tau^2 \|(I-I_N)\left( f^\prime (\Pi_N u(t_n))-f^\prime (\Pi_N u^{n}_{N}) \right)\cdot \Pi_N v(t_n)\|_{H^{-1}}\notag\\
		&\quad +\tau^{2} \| I_N f^\prime (\Pi_N u_N^n)( \Pi_N v(t_n) - \Pi_N v_N^n)\|_{H^{-1}} +\| \mathcal{L}^n\|_ 0\notag\\
		&=:\sum^{7}_{l=1}K^{n}_{l}+\| \mathcal{L}^n\|_ 0.
\end{align}
Since the $H^{-1}$ norm can be bounded by $L^{2}$ norm, it follows that
\begin{align*}
\sum_{l=1}^{4}K^{n}_{l}&\lesssim \tau \| f( \Pi_N u(t_n))- f( \Pi_N u_N^n)\|_ {L^{2}}+\tau \| (I-I_N)\left(f( \Pi_N u(t_n))- f( \Pi_N u_N^n)\right)\|_ {L^{2}} \\
&\lesssim \tau \|\Pi_N E^{n}_N\|_0+\tau N^{-1-}\|f( \Pi_N u(t_n))- f( \Pi_N u_N^n)\|_{H^{1+}}\\
&\lesssim \tau \|\Pi_N E^{n}_N\|_0+\tau N^{-\gamma+},
\end{align*}
where we have used the boundedness of $\|u_N^n\|_{H^{\gamma}}$ in Lemma~\ref{lem1} in the last inequality.
By \eqref{dualargu2} and \eqref{embedding2}, $K^n_5+K^n_7$ can be controlled as follows, 
\begin{align*}
K^n_5+K^n_7\lesssim&\tau^2 \| f^\prime( \Pi_N u(t_n)) -  f^\prime( \Pi_N u_N^n)\|_{L^2} \|  \Pi_N v(t_n)\|_{L^{2+}}\\
&+ \tau^2 \| I_N f^\prime (\Pi_N u_N^n)\|_{H^{1+}} \|( \Pi_N v(t_n) - \Pi_N v_N^n)\|_{H^{-1}}.
\end{align*}
Due to the boundedness of the trigonometric interpolation operator $I_N$ in $H^{s}$, $s>\frac{d}{2}$ and the estimate in Lemma \ref{lem1} we get $\| I_N f^\prime (\Pi_N u_N^n)\|_{H^{1+}}\lesssim \|f^\prime (\Pi_N u_N^n)\|_{H^{1+}}\lesssim N^{1-\gamma+}$. Thus we have
\begin{align*}
K^n_5+K^n_7\lesssim\tau^{2}\|\Pi_N E_N^{n}\|_ 0\left(N^{0+}\|  \Pi_N v(t_n)\|_{L^{2}}+\| I_N f^\prime (\Pi_N u_N^n)\|_{H^{1+}}\right)\lesssim\tau^{2}N^{1-\gamma+}\|\Pi_N E_N^{n}\|_ 0.
\end{align*}
To estimate $K^n_6$, we note from \eqref{embedding3} and Sobolev embedding $H^{0+}\hookrightarrow L^{2+}$ that
\begin{align*}
K^n_6\lesssim&\tau^2 \|(I-I_N)\left( f^\prime (\Pi_N u(t_n))-f^\prime (\Pi_N u^{n}_{N}) \right)\|^{\gamma}_{H^{0+}}\\
&\cdot\|(I-I_N)\left( f^\prime (\Pi_N u(t_n))-f^\prime (\Pi_N u^{n}_{N}) \right)\|^{1-\gamma}_{H^{1+}}\cdot \|\Pi_N v(t_n)\|_{H^{\gamma-1}}\\
\lesssim &\tau^{2}N^{-\gamma+}\|f^\prime (\Pi_N u(t_n))-f^\prime (\Pi_N u^{n}_{N}) \|_{H^{1+}}\cdot \|\Pi_N v(t_n)\|_{H^{\gamma-1}}\\
\lesssim &\tau^{2}N^{1-2\gamma+},
\end{align*}
where we have used Lemma~\ref{lemInterp} and Lemma~\ref{lem1} in the last two inequalities. 
 Then substituting the estimates on $K^n_i$, for $i=1,2,\cdots,7$ and \eqref{Ln0} into \eqref{estimate_Rn0}, we obtain
\begin{align*}
\|\mathcal{R}^n\|_0&\lesssim \tau\| E_N^{n}\|_ 0+ \tau^2 N^{1- \gamma+} \| E_N^n\|_0
+ \tau \kappa(N,\tau,\gamma),
\end{align*}
where we denote
\begin{align*}
\kappa(N,\tau,\gamma)=\left\{ \begin{array}{ll}
N^{- \gamma+} + \tau N^{1-2 \gamma+} ,\quad &\text{for} \quad \gamma\in(0,\frac{1}{2}],\\[2mm]
 N^{- \gamma+}+ \tau N^{1-2 \gamma+}+\min(\tau, \tau^2 N^{2(1- \gamma)+}),\quad &\text{for}\quad \gamma\in (\frac{1}{2},1].
\end{array}
\right.
\end{align*}
Under the assumption $N\le \tau^{-\frac{1}{1- \gamma}+\epsilon_{0}}$, by using Gronwall's inequality, we obtain
$$
 \left\|E_N^{n}\right\|_0 \lesssim 
 \|E_N^0\|_0 +\kappa(N,\tau,\gamma)=\|\Pi_{>N^{\alpha}}U^{0}\|_0+\kappa(N,\tau,\gamma),
 $$
which proves \eqref{eq:thm_2} and \eqref{eq:thm_1} when $\alpha\geq 1$.

\section{Proof of Theorem \ref{u11}}\label{section:proof2}

From the estimates of the remainders in Section \ref{section:proof} we can see that $\| R_1(t_n)\|_0$ and $\| R_8(t_n)\|_0 $ are $O( \tau N^{- \gamma+})$ while the other remainders are $O( \tau N^{-2\gamma+})$ under the step size condition $\tau\sim N^{-1}$. In the case $\gamma=\frac12-$, all the remainders except $\| R_1(t_n)\|_0$ and $\| R_8(t_n)\|_0 $ are $O( \tau^{2-})$, which leads to almost first-order convergence. Therefore, we only need to show the following improved error estimates for remainders $R_1(t_n)$ and $R_8(t_n)$ (see Lemma \ref{lem:improved_R1}  and Lemma \ref{lem:improved_R8}). 

\begin{lemma}\label{lem:improved_R1} 
Under the regularity condition $U\in C([0,T];H^{\gamma}(\Omega)\times H^{\gamma-1}(\Omega))$ with $\gamma\in(0,1]$, the remainder $R_1(t_n)$ defined in \eqref{defR1} satisfies the following improved estimate:
\begin{align}\label{eq:improved_R1}
\max_{0\le n\le M} \|R_1(t_n)\|_0\lesssim \tau N^{-2\gamma+} . 
\end{align}
\end{lemma} 

\begin{proof}
We decompose $R_1(t_n)$ into the following two parts: 
\begin{align}
R_1(t_n)=
\label{eq:3}
& \int_0^\tau e^{(\tau-s) L} {\Pi_{>N}} F (\Pi_N U\left(t_n+s\right) ) d s \\
\label{eq:4}
& +\int_0^\tau e^{(\tau-s) L}
\left[F\left(U\left(t_n+s\right)\right)-F\left(\Pi_N U\left(t_n+s\right)\right)\right] d s , 
\end{align}
where \eqref{eq:4} can be further decomposed dyadically as follows (with $ m = \lceil \log_2 N\rceil$ below): 
\begin{align*}
\eqref{eq:4} 
&= \int_0^ \tau e^{(\tau - s) L}\left[ 
	F(U(t_n+s)) - F( \Pi_{2^m N} U(t_n+s))
\right] ds\\
&\quad\, + \sum_{j=0}^{m-1} \int_0^ \tau e^{(\tau - s) L} 
\left[ 
	F( \Pi_{2^{j+1}N} U(t_n+s)) - F( \Pi_{2^{j}N} U(t_n+s))
\right] ds =: G_0 + \sum_{j=0}^{m-1} G_j . 
\end{align*}
By applying the mean-value theorem, we obtain
\begin{align*}
\| G_j\|_0 &\lesssim \tau \sup_{s\in [0, \tau]}\| F( \Pi_{2^{j+1}N} U(t_n+s)) - F( \Pi_{2^{j}N} U(t_n+s))\|_0\\
& \lesssim \tau \max_{s, \theta} 
\| f^\prime( \xi) 
\cdot (\Pi_{2^{j+1}N} - \Pi_{2^{j}N}) u(t_n+s) \|_{H^{-1}},
\end{align*}
where $ \xi = \theta\Pi_{2^{j+1}N} u(t_n+s) + (1- \theta) \Pi_{2^{j}N} u(t_n+s) $.
By estimating the last term above using \eqref{dualargu2}, we have 
\begin{align*}
\| G_j\|_0& \lesssim \tau \| (\Pi_{2^{j+1}N} - \Pi_{2^{j}N}) u(t_n+s) \|_{H^{-1}}
( \| f^\prime( \xi)\|_{H^{1+}}
	+ \| f^\prime\|_{L^\infty} )\\
&\lesssim \tau(2^j N)^{-1- \gamma } \| u\|_{L^\infty H^ \gamma} 
	\left( (2^j N)^{1- \gamma+ } \| u\|_{L^\infty H^ \gamma} +1\right)\le \tau(2^j N)^{-2 \gamma+} \| u\|^2_{L^\infty H^ \gamma}.
\end{align*}
For $G_0$, we have
\begin{align*}
\| G_0\|_{1}&\lesssim
\tau \max_{s\in (0, \tau), \theta\in (0,1]} 
\| f^\prime( \theta u(t_n+s) + (1- \theta) \Pi_{2^mN} u(t_n+s)) 
	\cdot \Pi_{>2^mN} u(t_n+s)\|_{L^2}\\
&\lesssim \tau  \max_{s\in (0, \tau)}\| \Pi_{>N^2} u(t_n+s)\|_{L^2}\lesssim \tau N^{-2 \gamma} \| u\|_{L^\infty H^ \gamma}.
\end{align*}
In summary, we have
\begin{align}\label{eq4eq}
\|\eqref{eq:4}\|_0 &\lesssim \tau N^{-2 \gamma} + \sum_{j = 0}^{m} \tau (2^j N)^{-2 \gamma+} \lesssim \tau N^{-2 \gamma+}.
\end{align}
On the other hand, Bernstein's inequality in Lemma~\ref{bound} implies that 
\begin{align}\label{eq3eq}
\| \eqref{eq:3}\|_0 
&\lesssim \tau N^{-1- \gamma} 
\max_{s\in[0,\tau]} \| F( \Pi_N U(t_n+s))\|_{1+ \gamma} \lesssim \tau N^{-1- \gamma} \| U\|_ \gamma.
\end{align}
Therefore, by collecting \eqref{eq4eq} with \eqref{eq3eq}, we obtain \eqref{eq:improved_R1} for $\gamma\in(0,1]$. 
\hfill\end{proof}

\begin{lemma}\label{lem:improved_R8}
Under the regularity condition $U\in C([0,T];H^{\frac{1}{2}-}(\Omega)\times H^{-\frac{1}{2}-}(\Omega))$ and $u\in L^{\infty}(0,T;L^{\infty}(\Omega)\cap BV(\Omega))$, the remainder term $R_8(t_n)$ defined in \eqref{defR8} satisfies the following estimate: 
\begin{align}\label{eq:improved_R8}
\|R_8(t_n)\|_0\lesssim \tau N^{-1+}, \quad \text{for each}\quad 0\leq n\leq M.
\end{align}
\end{lemma}

Note that $R_8(t_n)$ is generated by the use of trigonometric interpolation $I_N$ on the nonlinear function for the implementation of FFT (instead of using projection $\Pi_N$). 
In order to prove Lemma~\ref{lem:improved_R8}, we need to use the following results for BV functions (see, for example \cite[Section 5.3]{Ziemer2012})

\begin{lemma}
For $u\in BV(\Omega)$ and $\varepsilon>0$, we define $u_{\varepsilon}=Eu*\varphi_{\varepsilon}$ as the regularization of $u$ based on an extension operator $E:L^1(\Omega)\rightarrow L^1(\R^d)$ which is bounded from $W^{k,p}(\Omega)$ to $W^{k,p}(\R^d)$ for all $k\ge 0$ and $1\le p\le\infty$, and a mollifier $\varphi_{\varepsilon}$ defined on $\R^d$. Then, for all sufficiently small $\varepsilon>0$,
\begin{align}\label{BV_1}
 \lim_{\varepsilon\rightarrow 0}\int_{\Omega}|u_{\varepsilon}-u|dx=0,
\end{align}
and
\begin{align}\label{BV_2} 
\|u_\varepsilon\|_{BV(\Omega)}\lesssim \|u\|_{BV(\Omega)}.
\end{align}
Meanwhile, $u_\varepsilon\in C^{\infty}(\Omega)$ and $\|u_\varepsilon\|_{BV(\Omega)}=\|u_\varepsilon\|_{W^{1,1}(\Omega)}$, and $\|u_\varepsilon\|_{L^{\infty}(\Omega)}\leq \|u\|_{L^{\infty}(\Omega)}$. 
\end{lemma}

\begin{proof}[Proof of Lemma~\ref{lem:improved_R8}]
By recalling the definition of $R_8(t_n)$ in \eqref{defR8} and applying the 0 norm, we obtain
\begin{align}\label{improved_R8}
    \|R_8(t_n)\|_0\lesssim \tau \|( \Pi_N - I_N) f( \Pi_N u(t_n))\|_{H^{-1}}\lesssim \tau \|( \Pi_N - I_N) f( \Pi_N u(t_n))\|_{L^{1+}}.
\end{align}
In one dimension, we obtain the following estimate by using Lemma~\ref{bound} and Lemma~\ref{lemInterp_2}: 
\begin{align}\label{R8_0norm}
\|R_8(t_n)\|_0&\lesssim \tau \|( \Pi_N - I_N) f(\Pi_N u_\varepsilon(t_n))\|_{L^{1+}} 
    +\tau \|( \Pi_N - I_N) \left[f( \Pi_N u(t_n))-f(\Pi_N u_\varepsilon(t_n))\right]\|_{L^{1+}}\notag\\
    &\lesssim \tau N^{-1} \|f(\Pi_N u_\varepsilon(t_n))\|_{W^{1,1+}}+\tau N^{-1}\|f( \Pi_N u(t_n))-f(\Pi_N u_\varepsilon(t_n))\|_{W^{1,1+}}.
\end{align}
The first term in the right hand side of \eqref{R8_0norm} can be estimated by using the Lipschitz continuity of $f$ and the Sobolev embedding, i.e., 
\begin{align*}
&\tau N^{-1}\|f(\Pi_N u_\varepsilon(t_n))\|_{W^{1,1+}}\lesssim\tau N^{-1}\|\Pi_N u_\varepsilon(t_n)\|_{W^{1,1+}}\\
&\qquad\lesssim \tau N^{-1+} \| u_\varepsilon(t_n)\|_{W^{1-,1+}}
    \lesssim \tau N^{-1+} \| u_\varepsilon(t_n)\|_{W^{1,1}}= \tau N^{-1+} \| u_\varepsilon(t_n)\|_{BV}.
\end{align*}
The second term in the right hand side of \eqref{R8_0norm} can be estimated as follows: 
\begin{align*}
&\tau N^{-1}\|f( \Pi_N u(t_n))-f(\Pi_N u_\varepsilon(t_n))\|_{W^{1,1+}}\leq \tau N^{-1}\|f( \Pi_N u(t_n))-f(\Pi_N u_\varepsilon(t_n))\|_{L^{1+}}\\
&\qquad +\tau N^{-1}\|\nabla \Pi_N u(t_n)\left(f^{\prime}( \Pi_N u(t_n))-f^{\prime}(\Pi_N u_\varepsilon(t_n))\right)\|_{L^{1+}}\\
&\qquad+\tau N^{-1}\|f^{\prime}(\Pi_N u_\varepsilon(t_n))\left(\nabla \Pi_N u(t_n))-\nabla \Pi_N u^n_\varepsilon \right)\|_{L^{1+}}\\
&\lesssim \tau N^{\frac{1}{2}+}\|u(t_n)-u_\varepsilon(t_n)\|_{L^{1+}}.
\end{align*}
where the last inequality uses the Lipschitz condition of $f$ and the following result: 
\begin{align*}
&\tau N^{-1}\|\nabla \Pi_N u(t_n)\left(f^{\prime}( \Pi_N u(t_n))-f^{\prime}(\Pi_N u_\varepsilon(t_n))\right)\|_{L^{1+}}\\
&\lesssim \tau N^{-1}\|\nabla \Pi_N u(t_n) \|_{L^{\infty}}\|f^{\prime}( \Pi_N u(t_n))-f^{\prime}(\Pi_N u_\varepsilon(t_n))\|_{L^{1+}}\\
&\lesssim \tau N^{\frac{1}{2}+}\|u(t_n)\|_{H^{\frac{1}{2}-}}\|u(t_n)-u_\varepsilon(t_n)\|_{L^{1+}} .
\end{align*}
Combining these estimates, we have 
\begin{align}\label{improved_R8_1}
    \|R_8(t_n)\|_0\lesssim \tau N^{-1+} \| u_\varepsilon(t_n)\|_{BV}+\tau N^{\frac{1}{2}+}\|u(t_n)-u_\varepsilon(t_n)\|_{L^{1+}} . 
\end{align}
Under the assumption $u\in L^{\infty}([0,T]; L^{\infty}(\Omega))$, since $\|u_\varepsilon\|_{L^{\infty}}$ is uniformly bounded by $\|u\|_{L^{\infty}}$ and, according to \eqref{BV_1}, 
\begin{align*}
\|u(t_n)-u_\varepsilon(t_n)\|_{L^{1+}}\leq \|u(t_n)-u_\varepsilon(t_n)\|^{0+}_{L^{\infty}}\|u(t_n)-u_\varepsilon(t_n)\|^{1-}_{L^{1}}\rightarrow 0
\quad\mbox{as $\varepsilon\rightarrow 0$}. 
\end{align*}
Thus, for any fixed $N$, we can choose $\varepsilon\ll 1$ such that 
$\tau N^{\frac{1}{2}+}\|u(t_n)-u_\varepsilon(t_n)\|_{L^{1+}}\lesssim \tau N^{-1+}.$ 
This, together with \eqref{BV_2} and \eqref{improved_R8_1}, yields \eqref{eq:improved_R8}. 
This proves Lemma \ref{lem:improved_R8}. 
\end{proof}

By applying the improved error estimates in Lemma~\ref{lem:improved_R1} and Lemma~\ref{lem:improved_R8}, one get:

\begin{lemma}\label{lem:improved_L}
Under the conditions of Theorem~\ref{u11}, the remainder term $\mathcal{L}^{n}$ in \eqref{diff} satisfies the following improved estimate
\begin{align}\label{eq:improved_L}
\|\mathcal{L}^{n}\|_0\lesssim \tau^{2-},\quad \text{and}\quad \|\mathcal{L}^{n}\|_1\lesssim \tau^{\frac{3}{2}-},
\end{align}
for each $0\leq n\leq M$.
\end{lemma}

\begin{proof}
Under the conditions of Theorem~\ref{u11}, we combine the improved estimates Lemma~\ref{lem:improved_R1}, Lemma~\ref{lem:improved_R8} with \eqref{est:R2}, \eqref{est:R3}, \eqref{0_norm_R4_new}, \eqref{1_norm_I2}, and \eqref{0_norm_R9}, to get
\begin{align}\label{improved_L}
     \| \mathcal{L}^n\|_0=\sum_{i = 1}^9 \|R_i(t_n)\|_0\lesssim 
	\tau^2N^{0+}
         + \tau N^{-1+}\lesssim \tau^{2-},
\end{align}
under the step size condition $N=O(\tau^{-1})$.

On the other hand, recalling the 1-norm estimates on $R_i(t_n)$, $i=1, \cdots , 8$ in \eqref{est:R1}-\eqref{est:R3}, \eqref{1_norm_R4}, \eqref{1_norm_I2}, \eqref{est:R8} and the expression of $R_9(t_n)$ in \eqref{defR9} we obtain
\begin{align}\label{eq:Ln_H1}
\| \mathcal{L}^n\|_1&\leq \sum_{i=0}^{8}\|R_{i}(t_n)\|_{1}+\tau^2 \left(
 	\| ( \Pi_N - I_N) f( \Pi_N u(t_n))\|_{H^ 1} + 
	\|(I-I_N)f^\prime( \Pi_N u(t_n)) \Pi_N v(t_n)\|_{L^{2}}
 	\right).
\end{align}
By using the boundedness of $\Pi_N$ and $I_N$ in Lemma~\ref{bound} and Lemma~\ref{lemInterp}, there holds 
$$\| ( \Pi_N - I_N) f( \Pi_N u(t_n))\|_{H^ 1}\lesssim \| f( \Pi_N u(t_n))\|_{H^ 1}\lesssim N^{1-\gamma},$$
 and
\begin{align*}
\|(I-I_N)f^\prime( \Pi_N u(t_n)) \Pi_N v(t_n)\|_{L^{2}}&\lesssim \|(I-I_N)f^\prime( \Pi_N u(t_n))\|_{L^{2}}\|\Pi_N v(t_n)\|_{H^{\frac{1}{2}+}}\\
&\lesssim N^{-1}\|f^\prime( \Pi_N u(t_n))\|_{H^{1}}\|\Pi_N v(t_n)\|_{H^{\frac{1}{2}+}}\lesssim N^{\frac{3}{2}-2\gamma}.
\end{align*}
Thus substituting these estimates into \eqref{eq:Ln_H1}, we get
\begin{align*}
\| \mathcal{L}^n\|_1&\lesssim \tau(N^{-\gamma}+\tau N^{1-\gamma}+\tau)+\tau^{2}(N^{1-\gamma}+N^{\frac{3}{2}-2\gamma})\lesssim \tau^{\frac{3}{2}-},
\end{align*}
for $\gamma=\frac{1}{2}-$, under the step size condition $N=O(\tau^{-1})$. 
\end{proof}

We note from \eqref{estimate_Rn0} and the proof of Theorem~\ref{thm:convergence} that all the terms $K^{n}_j$ except $K^n_2$ are bounded by $O(\tau \|\Pi_N E^n_N\|_0+\tau^{2-})$, under the step size condition $\tau=O(N^{-1})$ when $\gamma=\frac{1}{2}-$. Therefore, in order to prove an improved error estimate for $K^n_2$, we need to first prove the following boundedness result for the projection error $\|\Pi_N U(t_n)-\Pi_N U^{n}_N\|_1$.

\begin{lemma}\label{lem:H1_error_bound}
Under the conditions in Theorem~\ref{u11}, there exists $\tau_0\in(0,1)$ such that for $\tau\in(0,\tau_0)$, we have the following estimate
\begin{align}\label{eq:H1_error_bound}
\|\Pi_N U(t_n)-\Pi_N U^{n}_N\|_1\lesssim \tau^{\frac{1}{2}-},
\end{align}
for each $0\leq n\leq M$.
\end{lemma}

\begin{proof}
Apply the projection operator $\Pi_N$ on both side of \eqref{sumid}, then we get
\begin{align}\label{projection_diff}
    \Pi_N E^{n+1}_N=\Pi_N\left(U(0)-U^{0}_N\right)+\sum_{j=0}^n e^{(n-j)\tau L} \Pi_N \mathcal{R}^j,
\end{align}
where $\mathcal{R}^n$ in \eqref{projection_diff} is given by \eqref{def-R^n}. 
Using the identity \eqref{identiy} and the definition of $H$ in \eqref{defH}, we have
\begin{align}\label{Rn_H1}
\|\mathcal{R}^n\|_1&	
	\lesssim \tau \| I_{N}\left[F( \Pi_N U(t_n))- F( \Pi_N U_N^n)\right]\|_ 1 
	+ \tau^2 \| I_N\left[f( \Pi_N u(t_n) ) - f( \Pi_N u_N^n)\right]\|_{H^1} \notag \\
	&\quad+ \tau^2 \| I_N f^\prime (\Pi_N u(t_n)) \cdot \Pi_N v(t_n)
			- I_N f^\prime (\Pi_N u_N^n) \cdot \Pi_N v_N^n\|_{L^2}
		+ \| \mathcal{L}^n\|_ 1 \notag\\
&=: \tilde{J}^n_1+\tilde{J}^n_2+\tilde{J}^n_3+\| \mathcal{L}^n\|_ 1.
\end{align}
Typically, by the Sobolev embedding $H^1\hookrightarrow L^{\infty}$ and the boundedness of the interpolation operator $I_N:H^{1}\rightarrow H^{1}$ in one dimensional, we can estimate $\tilde{J}^n_3$ by
\begin{align}\label{Rn3_H1}
\tilde{J}_3^{n}&\lesssim \tau^2 \Big(\| I_N [f^\prime (\Pi_N u(t_n))-  f^\prime (\Pi_N u_N^n)]\|_{L^{\infty}}\| \Pi_N v_N^n\|_{ L^2}\notag\\
&\qquad\qquad+\| I_N f^\prime (\Pi_N u(t_n)) \|_{L^{\infty}}\|\Pi_N (v(t_n)-v_N^n)\|_{ L^2}\Big)\notag\\
&\lesssim \tau^2 \Big(\| f^\prime (\Pi_N u(t_n))-  f^\prime (\Pi_N u_N^n)\|_{H^1}\|\Pi_N v(t_n)+\Pi_N (v_N^n-v(t_n))\|_{ L^2}\notag\\
&\qquad\qquad +\|  f^\prime (\Pi_N u(t_n)) \|_{H^1}\|\Pi_N (v(t_n)-v_N^n)\|_{ L^2}\Big)\notag\\
&\lesssim \tau^2 \|\Pi_N E^n_N\|_{1}(N^{\frac{1}{2}+}+\|\Pi_N E^n_N\|_{1}).
\end{align}
Then, regarding that the interpolation operator $I_N$ may be unbounded in $L^2$, we can estimate $\tilde{J}^n_1+\tilde{J}^n_2$ by the following decomposition:
\begin{align}\label{Rn1_Rn2_H1}
\tilde{J}^n_1+\tilde{J}^n_2&\leq \tau \|f( \Pi_N u(t_n) ) - f( \Pi_N u_N^n)\|_{L^2}+\tau\|(I-I_N)\left[f( \Pi_N U(t_n))- f( \Pi_N U_N^n)\right]\|_{L^2}\notag\\
&\quad+\tau^2 \| f( \Pi_N u(t_n) ) - f( \Pi_N u_N^n)\|_{H^1}\notag\\
&\lesssim \tau \|f( \Pi_N u(t_n) ) - f( \Pi_N u_N^n)\|_{L^2}+(\tau N^{-1}+\tau^2)\| f( \Pi_N u(t_n) ) - f( \Pi_N u_N^n)\|_{H^1}\notag\\
&\lesssim \tau \|\Pi_N E^n_N\|_{1}+(\tau N^{-1}+\tau^2) \|\Pi_N E^n_N\|_{1}(N^{\frac{1}{2}+}+\|\Pi_N E^n_N\|_{1})
\end{align}

Substituting \eqref{eq:improved_L}, \eqref{Rn3_H1} and \eqref{Rn1_Rn2_H1} into \eqref{Rn_H1}, we get
\begin{align*}
\|\mathcal{R}^{n}\|_1\lesssim \tau \|\Pi_N E^n_N\|_{1}+\tau^2 \|\Pi_N E^n_N\|_{1}(\tau^{-\frac{1}{2}-}+\|\Pi_N E^n_N\|_{1})+\tau^{\frac{3}{2}-},
\end{align*} 
under the step condition $N=O(\tau^{-1})$. This estimate together with \eqref{projection_diff} and the initial condition $\Pi_N U(0)-\Pi_N U^{0}_N=0$, yield
\begin{align}\label{Gronwall_H1}
\|\Pi_N E^{n+1}_N\|_1\leq C\sum_{j=0}^{n}\tau (\|\Pi_N E^j_N\|_{1}+\tau^{\frac{1}{2}-}\|\Pi_N E^j_N\|_{1}^2+\tau^{\frac{1}{2}-}),
\end{align}
which implies \eqref{eq:H1_error_bound} by using Gronwall's inequality.
\end{proof}

We now ready to prove the improved error estimate on the high frequency recovered low-regularity integrator \eqref{reformulated_scheme}.

\begin{proof}[Proof of Theorem~\ref{u11}]
By applying the projection operator $\Pi_N$ and $\Pi_{(N,N^2]}$ on \eqref{sumid}, we have
\begin{align}
    \Pi_N E^{n+1}_N&=\Pi_N\left(U(0)-U^{0}_N\right)+\sum_{j=0}^n e^{(n-j)\tau L} \Pi_N \mathcal{R}^j,\label{projection_diff_low}\\
    \Pi_{(N,N^2]} E^{n+1}_N&=\Pi_{(N,N^2]}\left(U(0)-U^{0}_N\right)+\sum_{j=0}^n e^{(n-j)\tau L} \Pi_{(N,N^2]} \mathcal{L}^j,\label{projection_diff_high}
\end{align}
where we have used the equality $\Pi_{(N,N^2]}\mathcal{R}^{j}=\Pi_{(N,N^2]}\mathcal{L}^{j}$ in \eqref{projection_diff_high}. Using the same estimate as in the proof of Theorem~\ref{thm:convergence}, we obtain
\begin{align}\label{0_norm_R^j}
\|\Pi_N \mathcal{R}^j\|_0\leq \|\mathcal{R}^j\|_0\lesssim \sum_{l=1}^{7}K^{j}_{l}+\|\mathcal{L}^{j}\|_0,
\end{align}
with $K^{j}_{l}$ defined in \eqref{estimate_Rn0}, and there holds 
\begin{align}
K^j_1+\sum_{l=3}^{7}K^j_l\lesssim \tau\|\Pi_N E^{j}_N\|_0+\tau^{2-}, 
\end{align}
under the step condition $N=O(\tau^{-1})$. For $K^{j}_2$, we apply the improved estimate in Lemma~\ref{lem:improved_R8} to get
\begin{align}\label{est:K^j_2}
K^{j}_2&\lesssim \tau\|(I-I_N)f(\Pi_N u(t_j))\|_{H^{-1}}+\tau\left\|(I-I_N)f\left(\Pi_N \big(u(t_j)+\Pi_N (u(t_j)-u^j_N)\big)\right)\right\|_{H^{-1}}\notag\\
&\lesssim \tau N^{-1+}\left(1+\|u(t_j)\|_{BV}+\|u(t_j)+\Pi_N (u(t_j)-u^j_N)\|_{BV}\right)\notag\\
&\lesssim \tau N^{-1+}\left(1+\|u(t_j)\|_{BV}+\|\Pi_N U(t_j)-\Pi_N U^{j}_N\|_{1}\right)\lesssim \tau N^{-1+},
\end{align}
where we have used $\|\Pi_N (u(t_j)-u^j_N)\|_{BV}\leq \|\Pi_N (u(t_j)-u^j_N)\|_{H^{1}}$ and Lemma~\ref{lem:H1_error_bound} in the last inequality.
Under the step size condition $N=O(\tau^{-1})$, one deduce from \eqref{0_norm_R^j}--\eqref{est:K^j_2} that
\begin{align*}
\|\Pi_N \mathcal{R}^j\|_0&\lesssim \tau\|\Pi_N E_N^{j}\|_ 0 +\| \Pi_N\mathcal{L}^j\|_0+ \tau^{2-}.
\end{align*}
This together with \eqref{projection_diff_low} and \eqref{projection_diff_high} implies
\begin{align}
\| \Pi_N E_N^{n+1}\|_ 0 & \leq
C_0\Big[\left\| \Pi_{N}\left(U^{0}-U^0_N\right)\right\|_ 0+\sum_{j=0}^n\left(\tau\| \Pi_N E_N^{j}\|_ 0 + \| \Pi_N\mathcal{L}^j\|_0+ \tau^{2-}\right)\Big],\label{low-frequency-estimate}\\
\| \Pi_{(N,N^2]} E_N^{n+1}\|_ 0 & \leq
C_0\Big[\left\| \Pi_{(N,N^2]}\left(U^{0}-U^0_N\right)\right\|_ 0+\sum_{j=0}^n\|\Pi_{(N,N^{2}]}  \mathcal{L}^j\|_0\Big].\label{high-frequency-estimate}
\end{align}
Since our choice of initial value $U_N^0=\Pi_{N^2}U(0)$ implies that 
$$\| \Pi_N (U(0)-U_N^0)\|_ 0=\| \Pi_{(N,N^{2}]} (U(0)-U_N^0)\|_ 0=0 , $$
by using the improved estimate on $\|\mathcal{L}^j\|_0$ in \eqref{eq:improved_L} and Gronwall's inequality, we obtain the following result from \eqref{low-frequency-estimate}--\eqref{high-frequency-estimate}: 
\begin{align}
\| \Pi_{N^{2}} U(t_{n+1})-U_N^{n+1} \|_ 0 
&\leq C\tau^{1-} .
\end{align}
This, together with an estimate of the projection error $\| U(t_{n+1}) - \Pi_{N^{2}} U(t_{n+1}) \|_ 0$ for $U^0\in H^{\frac12-}(\Omega)\times H^{-\frac12-}(\Omega)$, leads to the error estimate in Theorem \ref{u11}.
\end{proof}

From the proof of Theorem \ref{u11} we see that, without introducing the high-frequency recovery, the numerical solution satisfies  
$
\| \Pi_{N} U(t_{n+1})-U_N^{n+1} \|_ 0 \le C\tau^{1-}
$
but does not satisfy $\| U(t_{n+1})-U_N^{n+1} \|_ 0 \le C\tau^{1-}$ (therefore only has half-order convergence). 
This reduction of convergence rate (without the high-frequency recovery) can be observed in the numerical tests. 

\section{Numerical examples}\label{Section:Numerical}

In this section, we present extensive numerical examples to support the theoretical analysis and to illustrate the effectiveness of the low-regularity integrator in this paper in capturing the interface of discontinuity in the solutions, as well as the accuracy (without spurious oscillations) in approximating rough and discontinuous solutions of the semilinear wave equation.

\subsection{The Sine--Gordon equation in one dimension}\label{subsection:1d_Sine_Gordon}
We consider the semilinear wave equation in \eqref{KG_eq} with a nonlinear function $g(u)=40\sin(u)$ for the following piecewise smooth discontinuous initial state: 
\begin{align}\label{1d-initial-value}
	\big(u^{0}(x),v^{0}(x)\big)=\left\{
	\begin{array}{ll}
		{\displaystyle \left(5,-5\right)},\quad &\text{for } {\displaystyle x\in \big [0.3,0.425\big]},\\[1mm]
		{\displaystyle \left(2.5,-2.5\right)},\quad &\text{for } {\displaystyle x\in \big [0.575,0.7\big]},\\[1mm]
		{\displaystyle \left(0,0\right)},\quad &\text{else where} ,
	\end{array}
	\right. 
\end{align}
which satisfies the conditions of Theorem \ref{u11}. As a result, the low-regularity integrator \texttt{HR-LRI} in \eqref{reformulated_scheme} with $\alpha=2$ should have almost first-order convergence by choosing $N=O(\tau^{-1})$. We solve the problem with $N=2^{10}$ and $4\tau=N^{-1}$, and present the evolution of the numerical solution for $t\in[0,T]$ in Figure \ref{fig:4-1} (a), which clearly shows the propagation of discontinuities of the solution. 

\begin{figure}[htbp!]
	\centering
	\subfigure[Propagation of $u(t,x)$]{\includegraphics[width=6.3cm,height=5.3cm]{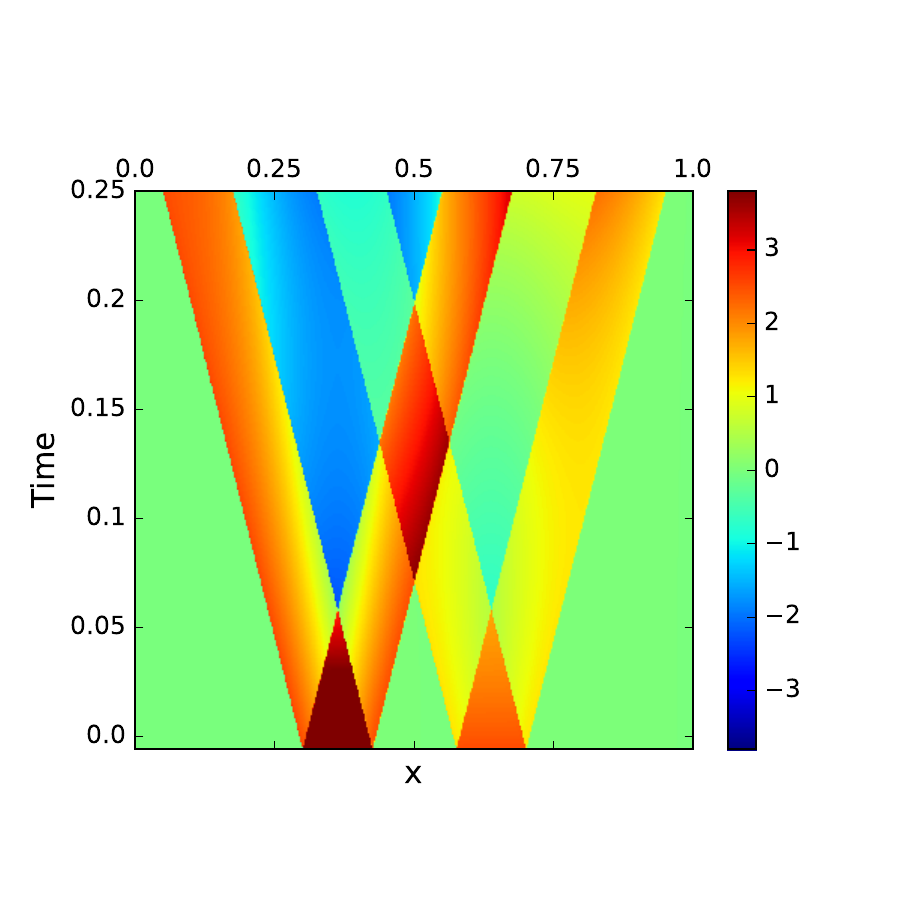}}
	\qquad
	\subfigure[Comparison of several different methods]{\includegraphics[width=6.3cm,height=5.3cm]{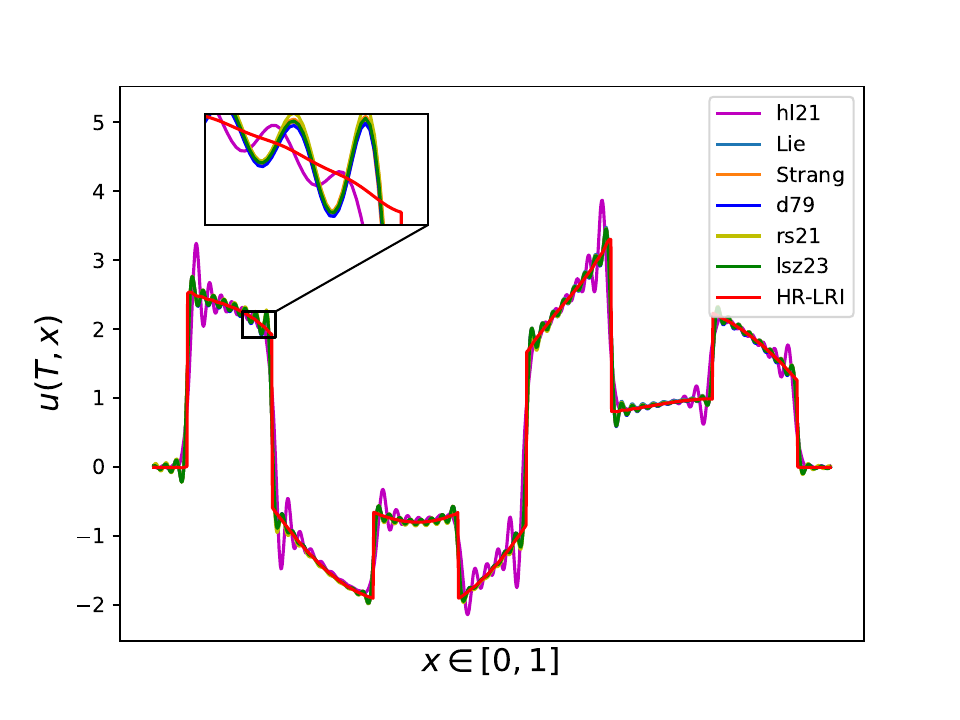}}
	\caption{Numerical solution of the 1D problem with discontinuous initial value in \eqref{1d-initial-value}. }
	\label{fig:4-1}
\end{figure}

\begin{figure}[htbp!]
	\centering
	\subfigure[$L^{2}(\Omega)\times H^{-1}(\Omega)$ error versus $\tau$]{\includegraphics[width=6.3cm,height=5.3cm]{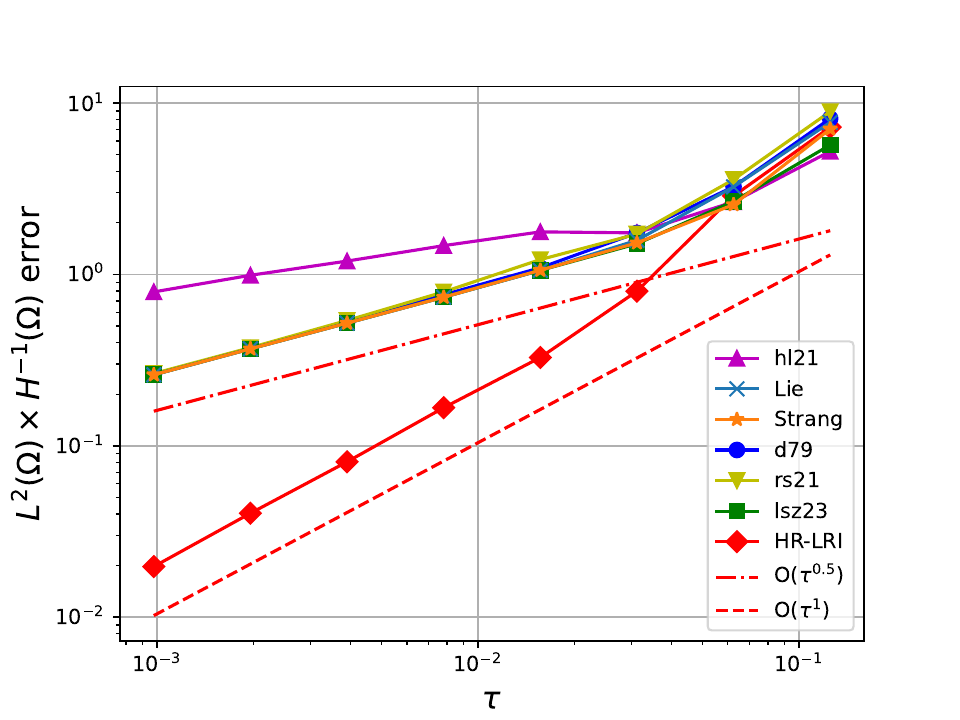}}
	\qquad
	\subfigure[$L^{2}(\Omega)\times H^{-1}(\Omega)$ error versus CPU time]{\includegraphics[width=6.3cm,height=5.3cm]{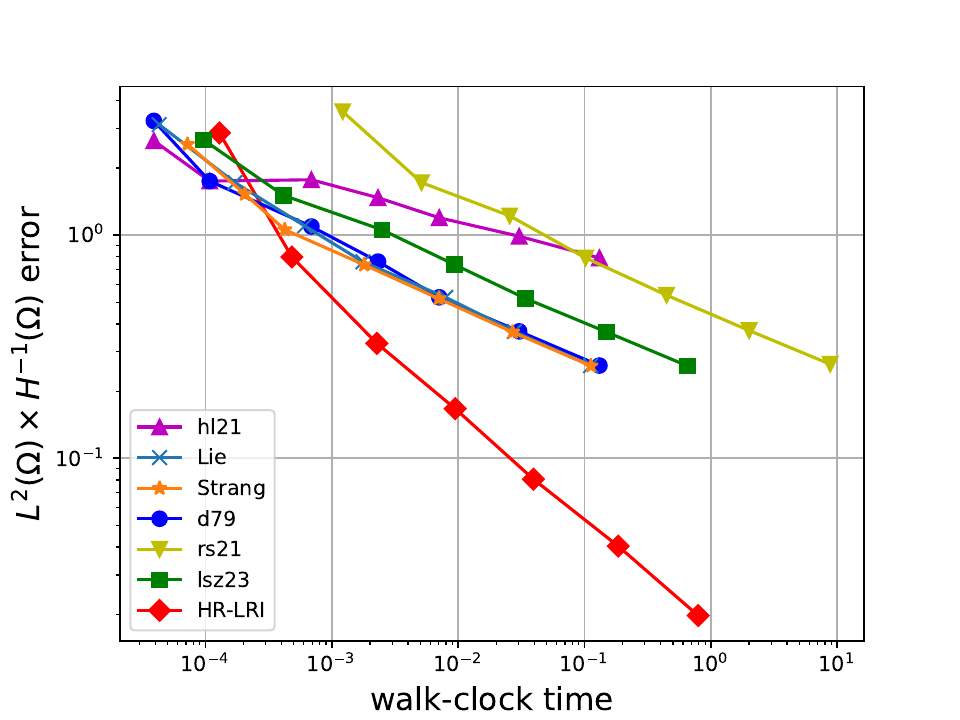}}
	\caption{Comparison of numerical solutions given by several different methods.}
	\label{fig:4-2}
\end{figure}

In Figure~\ref{fig:4-1} (b), we plot the numerical solutions at time $T$ computed by several different numerical methods, including the second-order low-regularity correction of Lie splitting method from \cite{LS2023} (which is referred to as \texttt{lsz23}), the second-order low-regularity exponential-type scheme from \cite{RS21} (which is referred to as \texttt{rs21}), the second-order IMEX method from \cite{HL2021} (referred to as \texttt{hl21}), the second-order trigonometric integrators constructed by Deuflhard \cite{D79} (referred to as \texttt{d79}), and classical splitting methods such as the Lie splitting scheme and Strang splitting scheme. The time step sizes and number of Fourier modes in all these methods are chosen to be $\tau=N^{-1}/4$ and $N=2^{7}$, respectively. 
From the numerical results in Figure~\ref{fig:4-1} (b) we can see that the discontinuities in the exact solution may lead to significant oscillations in the solutions of the pre-existing methods, while the proposed method in \texttt{HR-LRI} can substantially reduce the numerical oscillations with equivalent computational cost. 

In Theorem \ref{u11} and Remark \ref{resin} we have shown that, since the initial value of the solution is in $BV(\Omega)\cap L^{\infty}(\Omega)$, the error of the numerical solution is $O(\tau^{1-})$. In Figure~\ref{fig:4-2} we compare the $L^{2}(\Omega)\times H^{-1}(\Omega)$ errors of the numerical solutions at $T=0.25$ computed by the several different methods with step size $\tau=N^{-1}/4$ and the reference solutions are given by the proposed method with a sufficiently large $N=2^{14}$. The numerical results Figure~\ref{fig:4-2} are consistent with the theoretical results proved in Theorem \ref{u11} and demonstrates that the proposed method has a higher convergence rate (with respect to both step size and CPU time) than pre-existing methods in approximating discontinuous solutions of the semilinear wave equations. 

To further demonstrate the effectiveness of the proposed high-frequency recovery process, we compare the numerical solutions before and after high-frequency recovery at time $T=0.25$ in Figure~\ref{fig:4-10} (a). Here, \texttt{HR-LRI} represents the high-frequency recovery algorithm proposed in this paper, with $N=2^7$, $\tau=N^{-1}/4$, and $\alpha=2$. The ``\texttt{Without recovery}" corresponds to the algorithm \eqref{reformulated_scheme} without undergoing the high-frequency recovery process, with $N=2^7$ and $\tau=N^{-1}/4$. The reference solution in Figure~\ref{fig:4-10} (a) is obtained by using the proposed algorithm with sufficiently large $N=2^{14}$. And, in Figure~\ref{fig:4-10} (b), we compare the errors of numerical solutions before and after high-frequency recovery under  the condition $\tau=N^{-1}/4$. As rigorously proved by Theorem~\ref{u11}, the numerical experiments show that the high-frequency recovery process significantly reduces the oscillations of the solution, resulting in a higher order of convergence accuracy, whereas in the absence of the high-frequency recovery process, the numerical algorithm's order of convergence decreases significantly.
\begin{figure}[htbp!]
	\centering
	\subfigure[Comparison before and after high-frequency recovery process ]{\includegraphics[width=6.2cm,height=5.3cm]{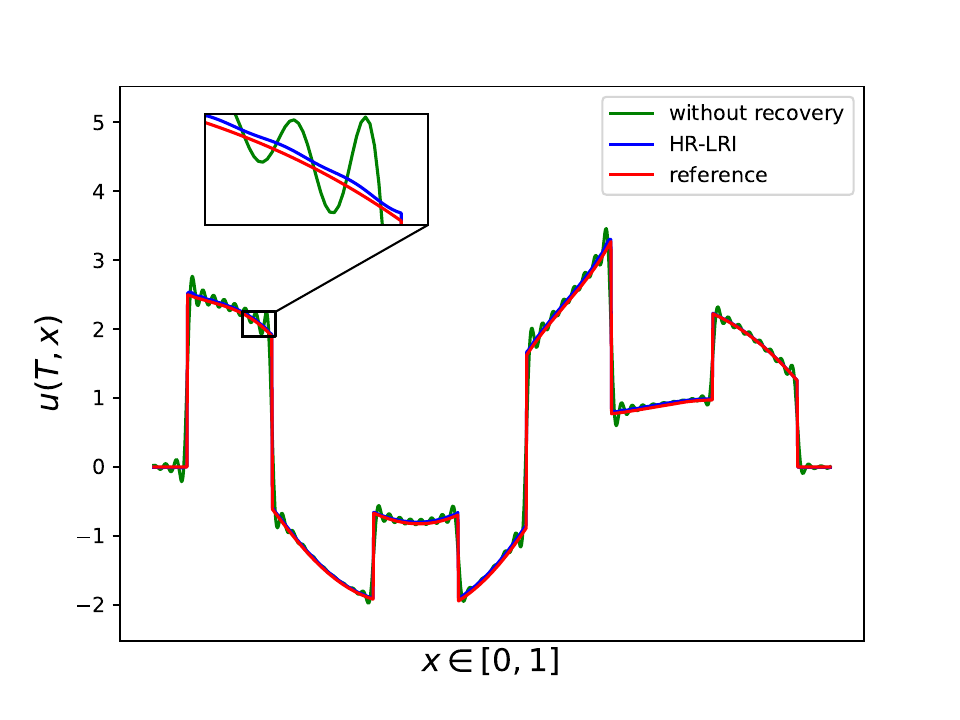}}
	\qquad
	\subfigure[$L^{2}(\Omega)\times H^{-1}(\Omega)$ error versus $\tau$]{\includegraphics[width=6.0cm,height=5.3cm]{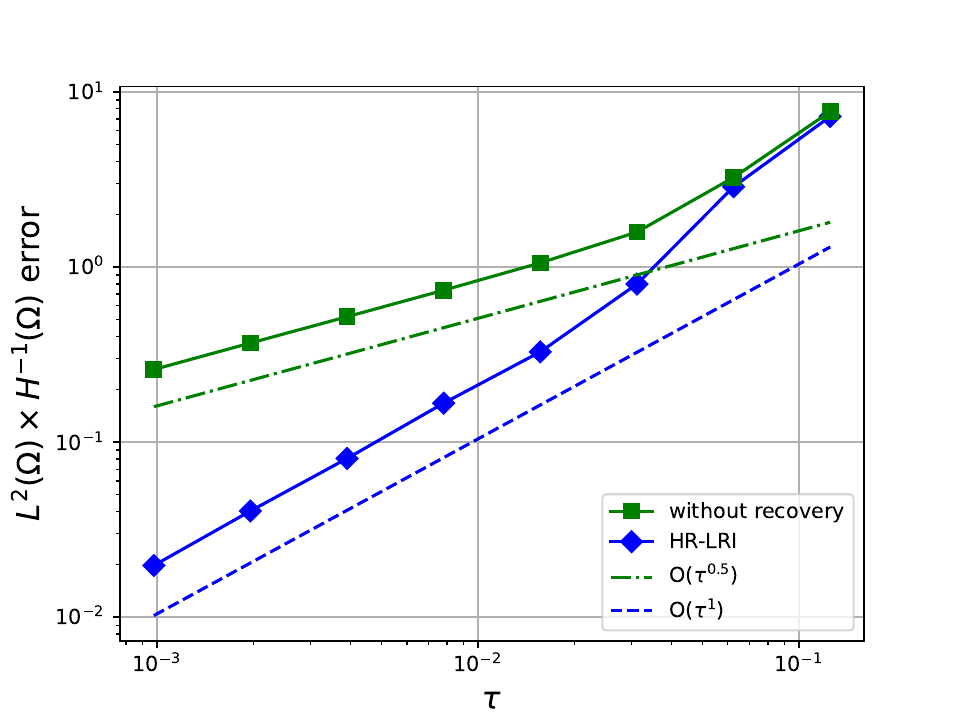}}
	\caption{Numerical results of the 1D problem with the initial value in \eqref{1d-initial-value}.}
	\label{fig:4-10}
\end{figure}

\subsection{The Sine--Gordon equation in two dimensions}
In this section, we consider the semilinear wave equation with $g(u)=4\sin(u)$ for the initial states:
\begin{itemize}
\item[(i)] Propagation of one discontinuous wave: 
$$\big(u^{0}(x),v^{0}(x)\big)=\left\{
\begin{array}{ll}
{\displaystyle \left(0.5,0\right)},\quad &\text{for } {\displaystyle x\in \big [0.375,0.625\big]^{2}},\\[4mm]
{\displaystyle \left(0,0\right)},\quad &\text{else where. }
\end{array}
\right. 
$$
\item[(ii)] Propagation of two discontinuous waves: 
$$\big(u^{0}(x),v^{0}(x)\big)=\left\{
\begin{array}{ll}
{\displaystyle \left(0.5,0\right)},\quad &\text{for } {\displaystyle x\in \big [0.3,0.425\big]^{2}},\\[4mm]
{\displaystyle \left(0.25,0\right)},\quad &\text{for } {\displaystyle x\in \big [0.575,0.7\big]^{2}},\\[4mm]
{\displaystyle \left(0,0\right)},\quad &\text{else where. }
\end{array}
\right. 
$$

\item[(iii)] 
A rough initial value in $H^{\frac12}\times H^{-\frac12}$: 
\begin{equation}\label{eq:rough_initial_2d}
(u^{0}(x),v^{0}(x))=\Big(\frac{1}{C_u}\sum_{k,l\in\mathbb{Z}}a_u(k)b_u(l)e^{i(kx_1+lx_2)},\;\frac{1}{C_v}\sum_{k,l\in\mathbb{Z}}a_v(k)b_v(l)e^{i(kx_1+lx_2)} \Big),
\end{equation}
where
\begin{align*}
\left\{
\begin{array}{ll}
a_u(k)=\text{rand}(0,1)|k|^{-1.01},\quad b_u(l)=\text{rand}(0,1)|l|^{-1.01},\\[2mm]
a_v(k)=\text{rand}(0,1)|k|^{-0.01},\quad b_v(l)=\text{rand}(0,1)|l|^{-0.01},
\end{array}
\right.
\end{align*}
and $C_u$ and $C_v$ are constants such that $\|u^{0}\|_{H^{
\frac{1}{2}}}=\|v^{0}\|_{H^{-\frac{1}{2}}}=1$.
\end{itemize}
We solve the semilinear wave equation by the Strang splitting method and the proposed low-regularity integrator \texttt{HR-LRI} in \eqref{reformulated_scheme} with $\alpha=\frac{3}{2}$ for the initial values in (i) and (ii), and plot the numerical solutions in Figure~\ref{fig:4-5-1}--\ref{fig:4-5} by choosing $\tau=N^{-1}/4=2^{-8}$. The results show that the proposed method can effectively eliminate the high oscillation of numerical solutions in approximating discontinuous solutions of the semilinear wave equation. 

\begin{figure}[!htbp]
\subfigure[Strang splitting]{\includegraphics[width=6.3cm,height=4.3cm]{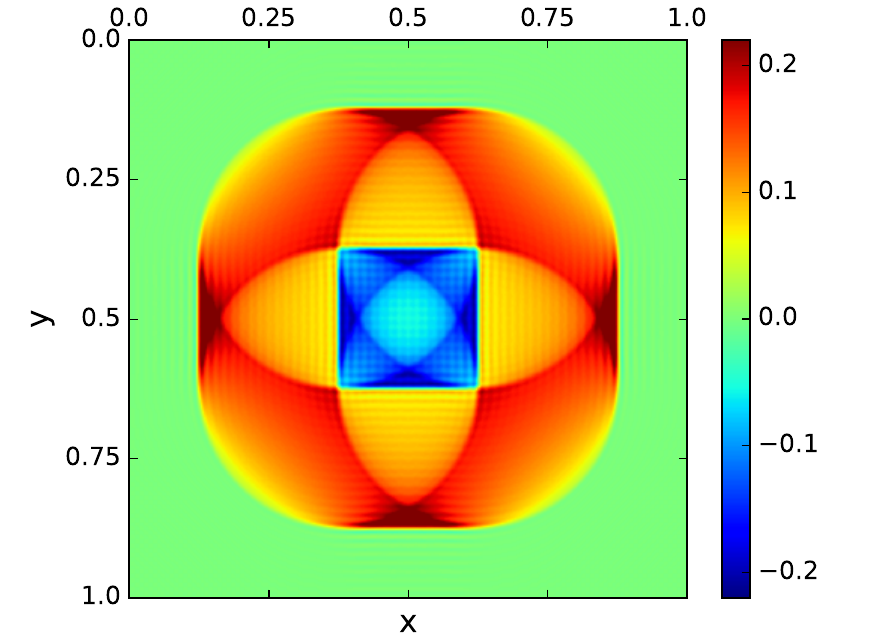}}
\qquad
\subfigure[The proposed method (\texttt{HR-LRI})]{\includegraphics[width=6.3cm,height=4.3cm]{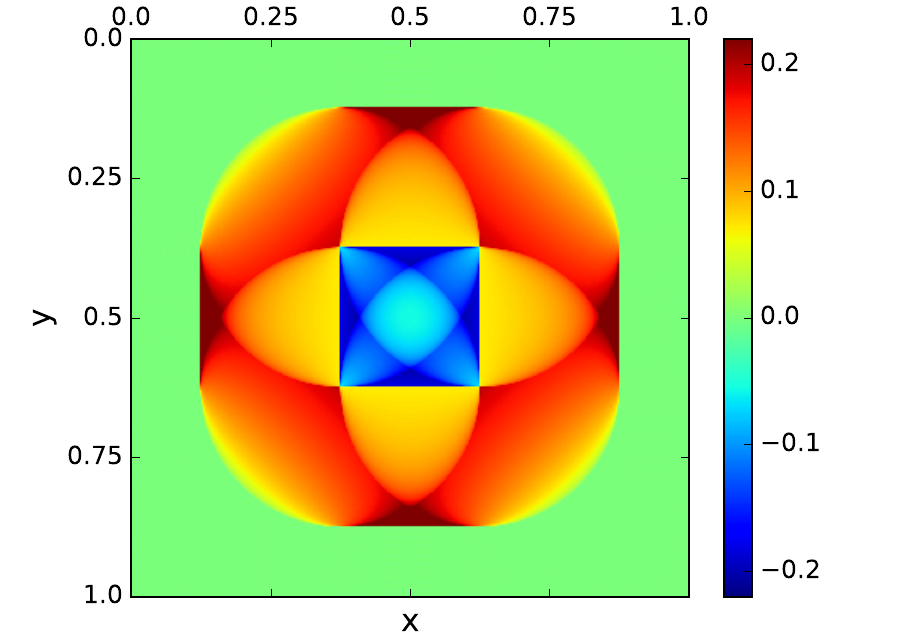}}
\caption{Comparison of the numerical solutions at $t=0.25$ computed by two different methods}
\label{fig:4-5-1}
\end{figure}

\begin{figure}[!htbp]
\subfigure[Strang splitting]{\includegraphics[width=6.3cm,height=4.2cm]{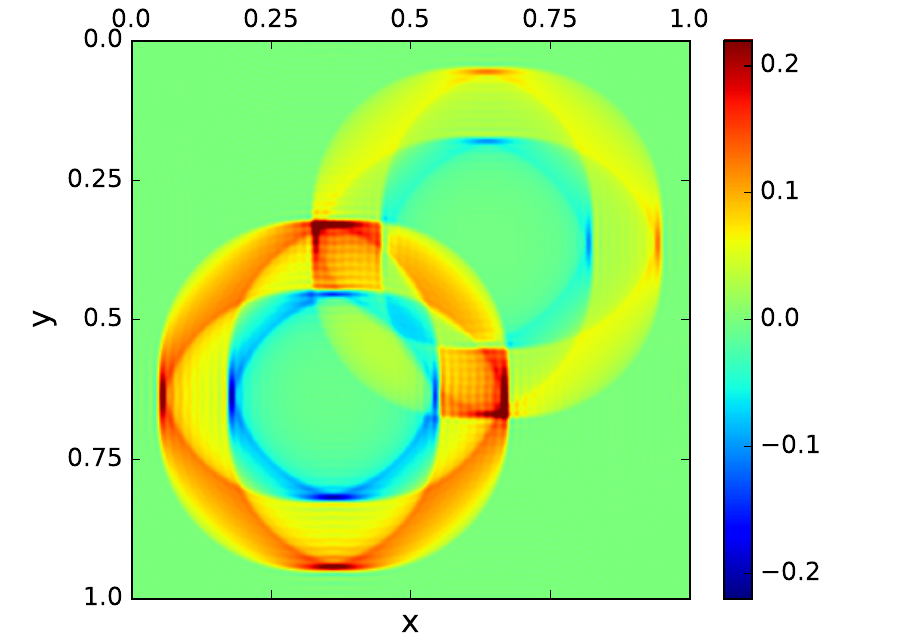}}
\qquad
\subfigure[The proposed method (\texttt{HR-LRI})]{\includegraphics[width=6.3cm,height=4.2cm]{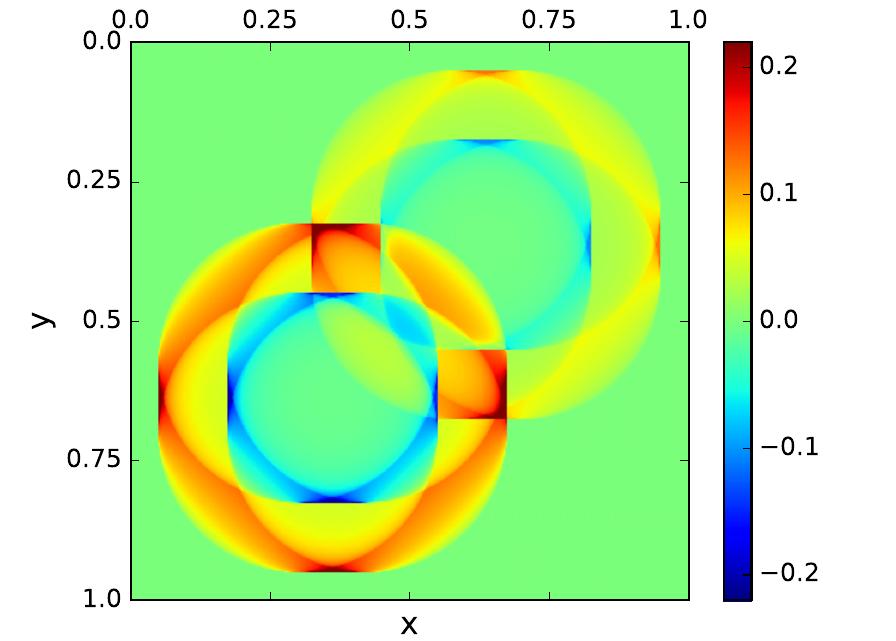}}
\caption{Comparison of the numerical solutions at $t=0.25$ computed by two different methods}
\label{fig:4-5}
\end{figure}



\begin{figure}[!htbp]
\centering
\subfigure[Error vs $\tau$ for initial state (i)]{\includegraphics[width=6.3cm,height=5.1cm]{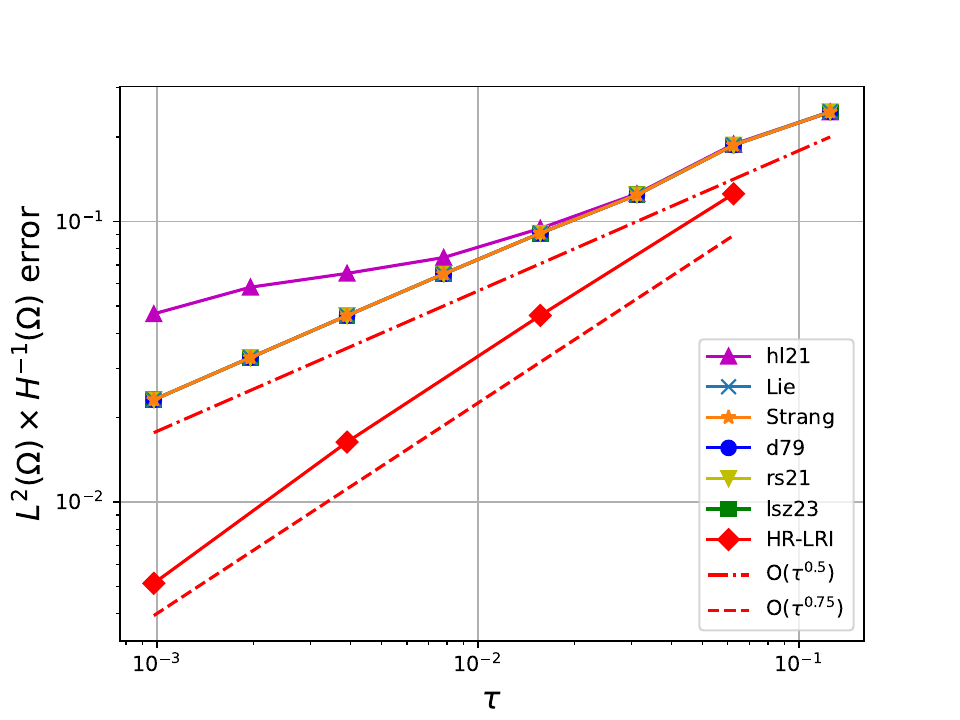}}
\qquad
\subfigure[Error vs CPU time for initial state (i)]{\includegraphics[width=6.3cm,height=5.1cm]{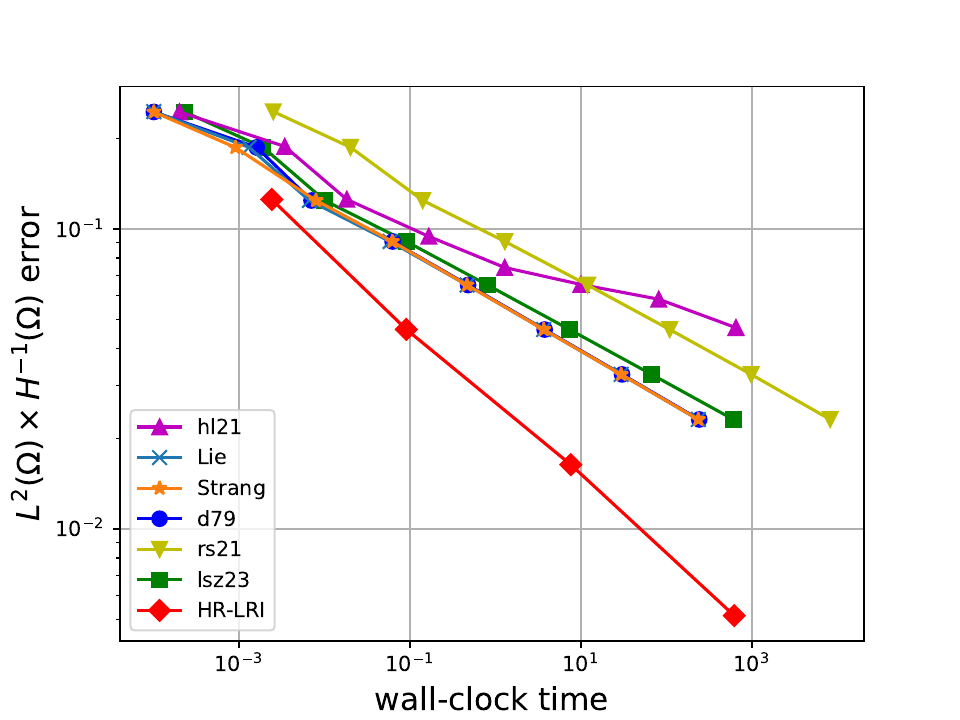}}
\vfill
\centering
\subfigure[Error vs $\tau$ for initial state (iii)]{\includegraphics[width=6.3cm,height=5.1cm]{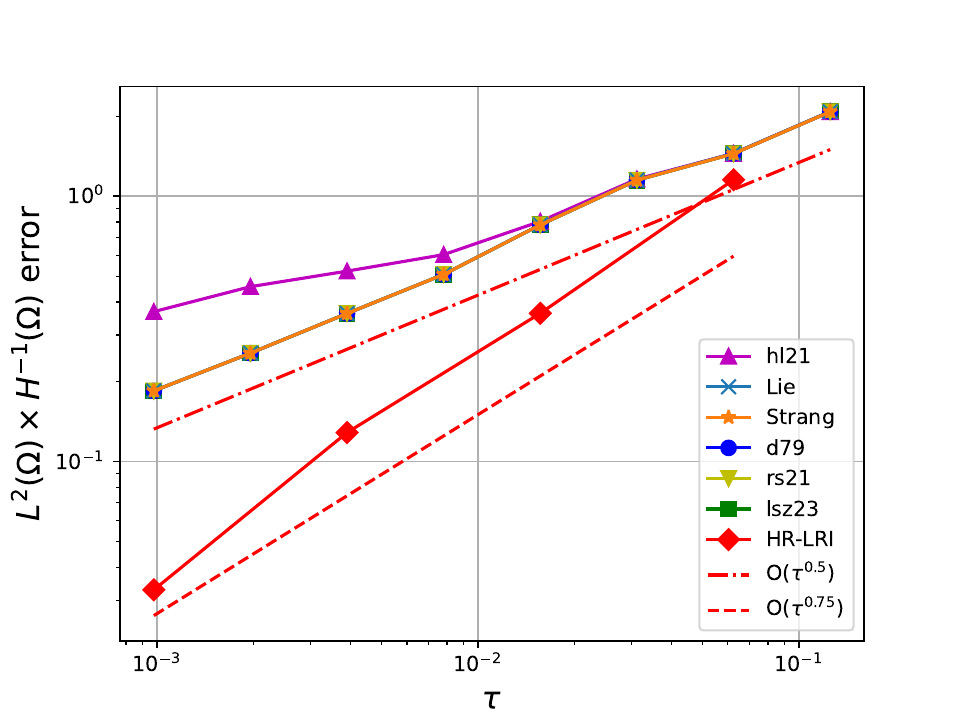}}
\qquad
\subfigure[Error vs CPU time for initial state (iii)]{\includegraphics[width=6.3cm,height=5.1cm]{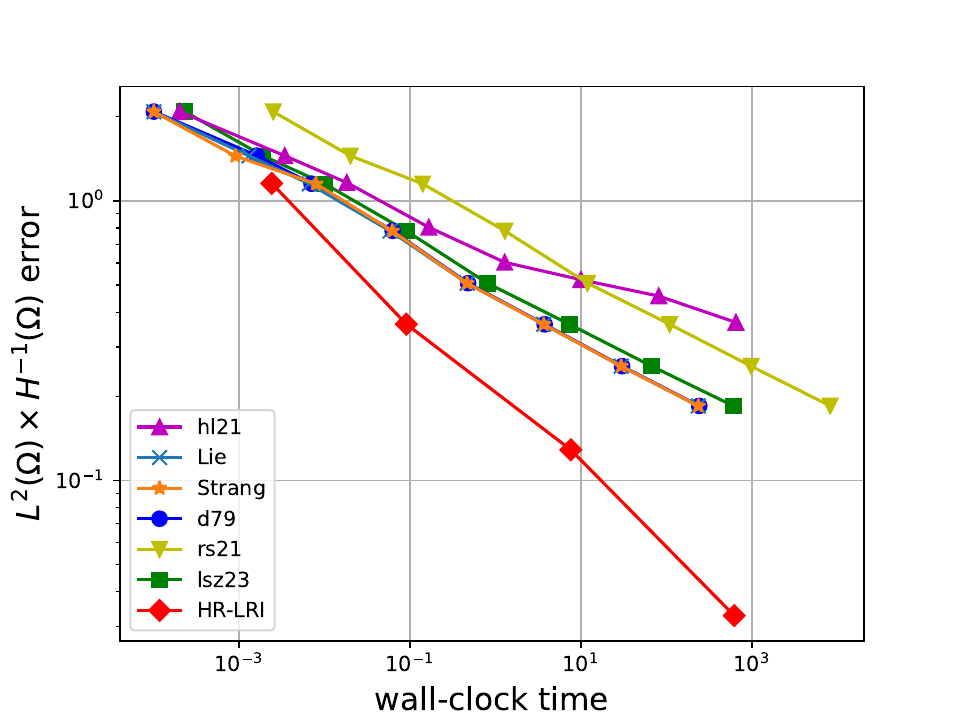}}
\caption{Comparison of numerical solutions given by several different methods in two dimensional cases.}
\label{fig:4-6}
\end{figure}

\pagebreak
The errors of the numerical solutions computed by several different numerical methods are presented in Figure \ref{fig:4-6} with $4\tau=N^{-1}$, where the reference solution is given by the proposed method with sufficiently large $N$ and sufficiently small $\tau$. The numerical results in Figure \ref{fig:4-6} are again consistent with the theoretical results proved in Theorem~\ref{thm:convergence}, i.e., the proposed method has convergence order $\frac34$ while the other methods have convergence order $\frac12$ or smaller. The convergence rate of the proposed method with respect to CPU time is also faster than pre-existing methods in approximating the discontinuous solution of the two-dimensional semilinear wave equation.


\subsection{The Klein--Gordon equation in one dimension}
In the last example, we consider the one-dimensional Klein--Gordon equation with a locally Lipschitz continuous (not globally Lipschitz continuous) nonlinear function $g(u)= u^{3}$, with the following piecewise smooth discontinuous initial state: 
\begin{align}\label{KG}
\big(u^{0}(x),v^{0}(x)\big)
=\left\{
\begin{array}{ll}
{\displaystyle \left(4,0\right)},\quad &\text{for } {\displaystyle x\in \big [0.3,0.425\big]},\\[4mm]
{\displaystyle \left(2,0\right)},\quad &\text{for } {\displaystyle x\in \big [0.575,0.7\big]},\\[4mm]
{\displaystyle \left(0,0\right)},\quad &\text{else where}, 
\end{array}
\right. 
\end{align}
which leads to a bounded piecewise smooth discontinuous solution. 
Since the Lipschitz continuity condition \eqref{cond:f} can be satisfied when $u(t,x)$ is uniformly bounded for $(t,x)\in [0,T]\times \Omega$, the theoretical results in Theorem~\ref{thm:convergence} and Theorem~\ref{u11} are also applicable to this problem. 


\begin{figure}[htbp!]
\centering
\subfigure[Evolution of numerical solution]{\includegraphics[width=6.5cm,height=5.5cm]{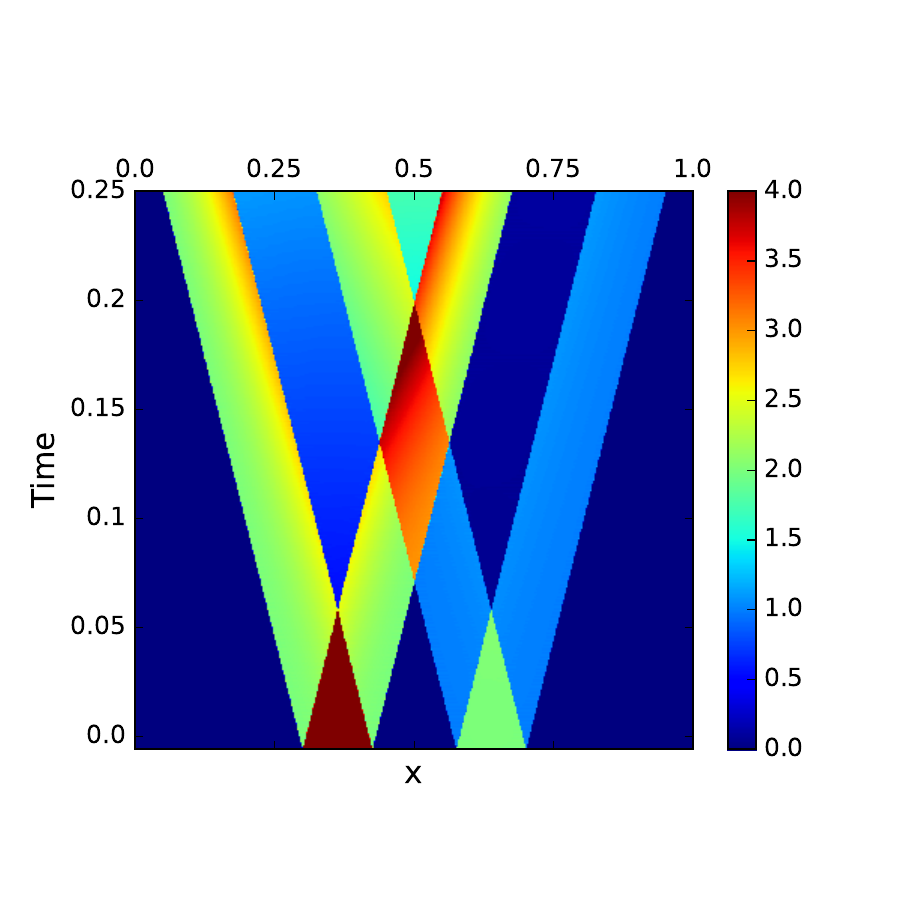} }
\qquad
\subfigure[Comparison of several different methods]{\includegraphics[width=6.0cm,height=5.3cm]{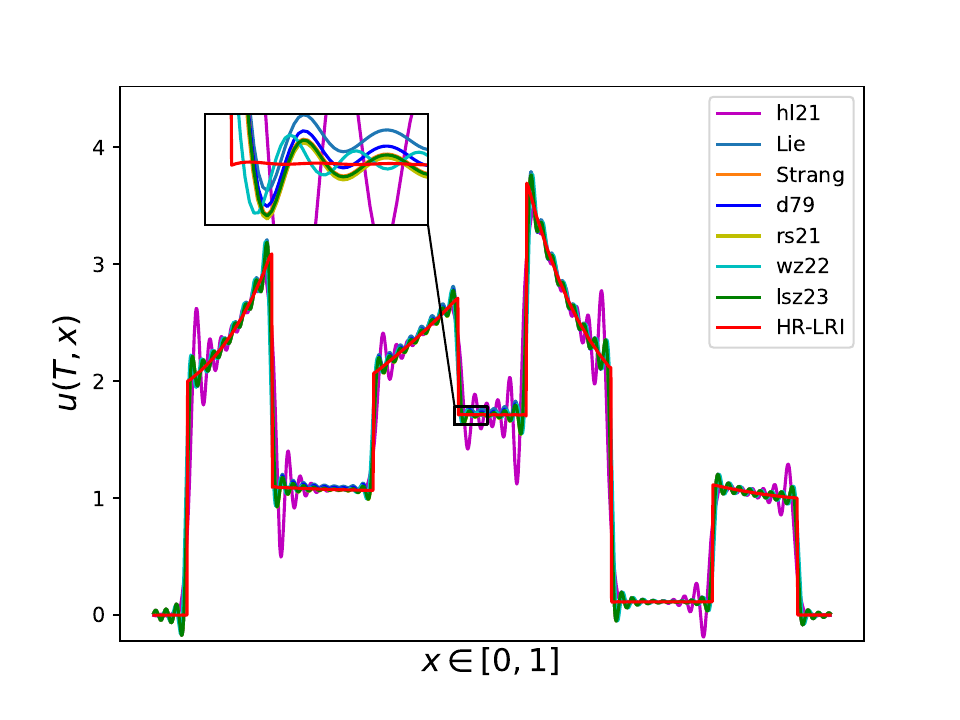} }
\caption{Numerical solution with the initial value in \eqref{KG}.}
\label{fig:4-8}
\end{figure}

\begin{figure}[htbp!]
\centering
\subfigure[$L^{2}(\Omega)\times H^{-1}(\Omega)$ error versus $\tau$]{\includegraphics[width=6.2cm,height=5.3cm]{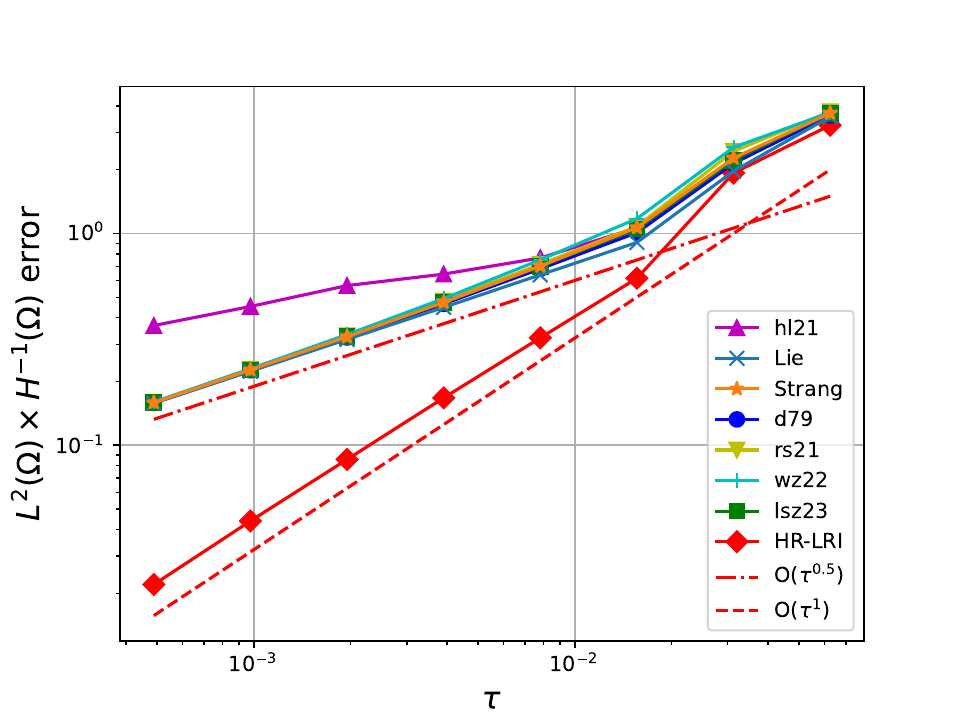}}
\qquad
\subfigure[$L^{2}(\Omega)\times H^{-1}(\Omega)$ error versus walk-clock time]{\includegraphics[width=6.0cm,height=5.3cm]{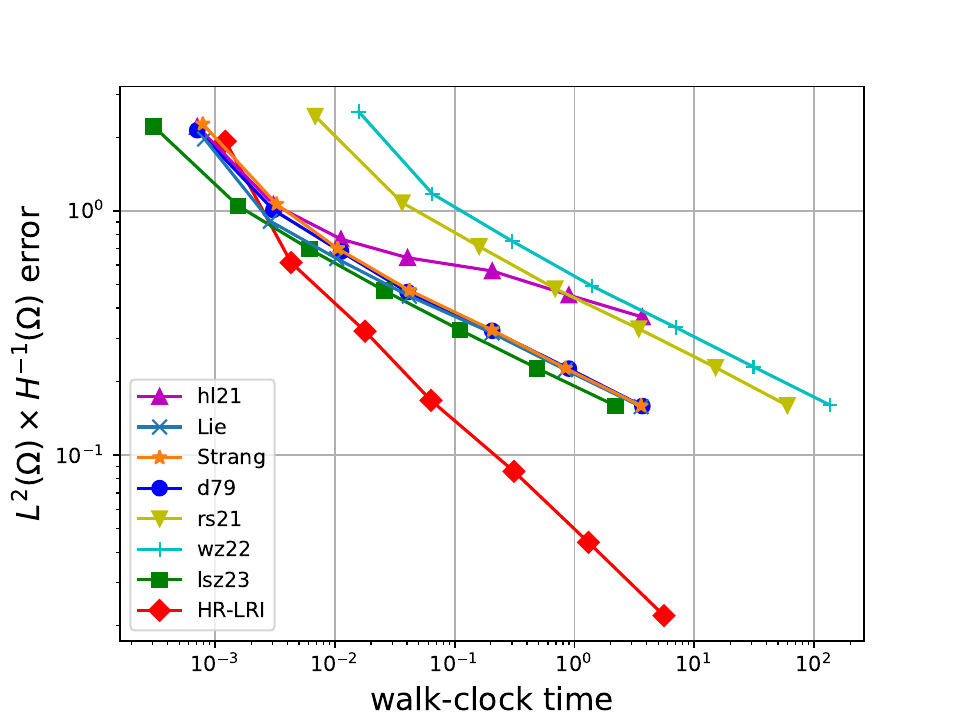}}
\caption{Errors of the numerical solutions with the initial value in \eqref{KG}.}
\label{fig:4-9}
\end{figure}

We solve the Klein--Gordon equation with $\tau=N^{-1}/4=2^{-12}$ and present the evolution of the numerical solution in Figure \ref{fig:4-8} (a), which shows the propagation of discontinuities of the solution. 
In Figure~\ref{fig:4-8} (b) we compare the numerical solutions at $T=0.25$ computed by several different methods. Here the symmetric low-regularity integrator for semilinear Klein--Gordon equation in \cite{WZ2022} is also taken into comparison and referred to as \texttt{wz22}. The numerical results in Figure~\ref{fig:4-8} (b) indicate that the proposed method indeed improves the accuracy and reduces spurious oscillations. 

In addition, we present the errors of the numerical solutions computed by the several different methods in Figure~\ref{fig:4-9}, which shows that the proposed method has first-order convergence with respect to the step size $\tau$ in approximating such discontinuous solutions, and the usual methods have half-order convergence in this case. 
This is consistent with the convergence rate proved in Theorem~\ref{u11} and demonstrates the efficiency of the proposed method in approximating rough and discontinuous solutions.


\end{document}